\documentclass[review,onefignum,onetabnum]{siamart190516}

\usepackage{lipsum}
\usepackage{amsfonts}
\usepackage{graphicx}
\usepackage{epstopdf}
\usepackage{algorithmic}

\usepackage{amssymb}
\usepackage{array,amsmath,amssymb,relsize,url,array,indentfirst,booktabs,geometry,}
\usepackage{listings}
\usepackage{verbatim}
\usepackage[caption=false]{subfig}
\usepackage{bm}

\newcommand{\norm}[1]{\left\lVert#1\right\rVert}

\def\<{{\langle}}
\def\>{{\rangle}}

\def\R{{\mathbb{R}}}

\newcommand{\bs}[1]{\boldsymbol{#1}}
\newcommand{\bit}{\begin{itemize}}
\newcommand{\eit}{\end{itemize}}

\newcommand{\ipar}{m}

\newcommand{\bP}{\mathbf{P}}
\newcommand{\bH}{\mathbf{H}}

\newcommand{\prmean}{\ipar_{\text{pr}}}

\newcommand{\pilike}{\pi_{\text{like}}}

\newcommand{\Gnoise}{\mathbf{\Gamma}_{\text{n}}}

\newcommand{\obs}{\mathbf{y}}
\newcommand{\qoi}{\boldsymbol{\rho}}

\newcommand{\logdet}[1]{\operatorname{logdet}\left( {#1} \right)}
\newcommand{\ave}{\mathbb{E}}

\newcommand{\bipar}{\mathbf{\ipar}}
\newcommand{\bF}{\mathbf{F}}

\newcommand{\bpostcov}{\mathbf{\Gamma}_{\text{post}}}

\newcommand{\bCpr}{\mathbf{\Gamma}_{\text{pr}}}
\newcommand{\bHmisfit}{\mathbf{H}_m}
\newcommand{\bpHmisfit}{\widetilde{\mathbf{H}}_m}

\usepackage{xcolor}

\renewcommand{\normalsize}{\fontsize{10pt}{\baselineskip}\selectfont}
\usepackage{textcomp,fancyhdr,lastpage}
\usepackage{extarrows}
\usepackage[utf8]{inputenc}
\DeclareMathOperator*{\argmax}{arg\,max}

\ifpdf
  \DeclareGraphicsExtensions{.eps,.pdf,.png,.jpg}
\else
  \DeclareGraphicsExtensions{.eps}
\fi

\newsiamremark{remark}{Remark}
\newsiamremark{hypothesis}{Hypothesis}
\crefname{hypothesis}{Hypothesis}{Hypotheses}
\newsiamthm{claim}{Claim}

\headers{An efficient method for GOOED}{K. Wu, P. Chen, and O. Ghattas}

\title{An efficient method for goal-oriented linear Bayesian optimal experimental design: Application to optimal sensor placement \thanks{Submitted to the editors DATE.
\funding{This research was partially funded by DOE ASCR 
  DE-SC0019303 and DE-SC0021239, DOD MURI FA9550-21-1-0084, 
  and NSF DMS-2012453.}}}

\author{Keyi Wu\thanks{Department of Mathematics, The University of Texas at Austin, TX 
  (\email{keyiwu@math.utexas.edu}).}
\and Peng Chen\thanks{Oden Institute for Computational Engineering and Sciences, The University of Texas at Austin, TX 
  (\email{peng@oden.utexas.edu},\email{omar@oden.utexas.edu}).}
\and Omar Ghattas\footnotemark[3]}

\usepackage{amsopn}

\ifpdf
\hypersetup{
  pdftitle={An efficient method for goal-oriented linear Bayesian optimal experimental design: Application to optimal sensor placement},
  pdfauthor={K. Wu, P. Chen, and O. Ghattas}
}
\fi

\begin{document}

\maketitle

% REQUIRED
\begin{abstract}
Optimal experimental design (OED) plays an important role in the problem of identifying uncertainty with limited experimental data. In many applications, we seek to minimize the uncertainty of a predicted quantity of interest (QoI) based on the solution of the inverse problem, rather than the inversion model parameter itself. In these scenarios, we develop an efficient method for goal-oriented optimal experimental design (GOOED) for large-scale Bayesian linear inverse problem that finds sensor locations to maximize the expected information gain (EIG) for a predicted QoI. By deriving a new formula to compute the EIG, exploiting low-rank structures of two appropriate operators, we are able to employ an online-offline decomposition scheme and a swapping greedy algorithm to maximize the EIG at a cost measured in model solutions that is independent of the problem dimensions. We provide detailed error analysis of the approximated EIG, and demonstrate the efficiency, accuracy, and both data- and parameter-dimension independence of the proposed algorithm for a contaminant transport inverse problem with infinite-dimensional parameter field.
\end{abstract}

% REQUIRED

% REQUIRED

\section{Introduction}
Optimizing the acquisition of data---e.g., what, where, and when to measure, what experiments to run---to maximize information gained from the data 
% subject to budget or time constraint 
is a fundamental and ubiquitous problem across all of the natural and social sciences, engineering, 
medicine, and technology. Just three important examples include optimal observing system design for ocean climate data \cite{LooseHeimbach21}, optimal sensor placement for early warning of tsunami waves \cite{FerrolinoLopeMendoza20}, and optimal experimental design to accelerate MRI imaging  \cite{BakkerHoofWelling20}. Bayesian optimal experimental design (BOED)---including formulations as active learning, Bayesian optimization, and sensor placement---provides a probabilistic framework to maximize the expected information gain (EIG) or mutual information (MI) for uncertain parameters or related quantities of interest \cite{ChalonerVerdinelli95}. However, evaluating the EIG remains prohibitive for large-scale, complex models, due to the need to compute double integrals with respect to both parameter and data distributions. Recently, advances in efficiently evaluating the EIG and optimizing the design have been achieved using methods based on posterior Laplace approximation-based EIG estimation \cite{LongScavinoTemponeEtAl13}, myopic posterior sampling for adaptive goal-oriented BOED \cite{KandasamyNeiswangerZhangEtAl19}, EIG estimation by variational inference for BOED \cite{FosterJankowiakBinghamEtAl19}, BOED for implicit models by neural EIG estimation \cite{KleinegesseGutmann20}, and sequential BOED with variable cost structure \cite{ZhengHaydenPachecoEtAl20}.

Interest has intensified in extending BOED to the case of  experiments on, or observations of, complex physical systems, since these can be very expensive (e.g., satellite trajectories, subsurface wells, ocean-bottom acoustic sensors). Such physical systems are typically modeled by partial differential equations (PDEs), which are expensive to solve and often contain infinite-dimensional parameter fields and large numbers of design variables. This presents fundamental challenges to conventional BOED methods, which require prohibitively large numbers of (PDE) model solves.
% In many practical problems, we use experimental (observation) data to gain information, identify the uncertainty of a quantity of interest. The acquisition of data is usually an expensive process due to limited budget or time considerations. Thus it is important to design the experiments wisely that enables us to gain most information from the data. i.e., to infer the quantity of interest accurately -this is the Bayesian optimal experimental design (OED) problem. Specifically, in the paper, we consider, where to place the sensors to observe the data that can optimally infer the quantity of interest for a large scale model.  \\
% To define the optimal design, we use expected information gain as the maximization criterion. The expected information gain (EIG) is a common choice of criterion measured by the Kullback-Leibler divergence between the prior and posterior over all realizations of the experimental data. For large-scale problems, even just to compute this quantity requires repeated evaluations of functionals of high-dimensional and expensive-to-apply operators involving complex forward model solves. For the linear inverse problem, EIG is equivalent to the D-optimal design, which measures the log-determinant of the prior-covariance preconditioned misfit Hessian.\\
Several different classes of methods have been developed to tackle these computational challenges by exploiting (1) sparsity by polynomial chaos approximation of parameter-to-observation maps \cite{HuanMarzouk13,HuanMarzouk14,HuanMarzouk16}, (2) intrinsic low dimensionality by low-rank approximation of (prior-preconditioned and data-informed) operators \cite{AlexanderianPetraStadlerEtAl14, AlexanderianGloorGhattas16,AlexanderianPetraStadlerEtAl16, SaibabaAlexanderianIpsen17,CrestelAlexanderianStadlerEtAl17,AttiaAlexanderianSaibaba18}, and (3) decomposibility by offline (for model-constrained EIG approximation)--online (for design optimization) decomposition \cite{WuChenGhattas20}.

%In particular, for linear PtO maps, \cite{AlexanderianGloorGhattas16} established an infinite-dimensional form of Kullback-Leibler divergence and EIG framework. With a closed form expressions of D-optimal criterion involving numerous inverse problem solves and eigenvalue decompositions of prior-preconditioned Hessian,  \cite{	Bui-ThanhGhattas12a,
%	Bui-ThanhGhattasMartinEtAl13,
%	AlexanderianPetraStadlerEtAl16, AlexanderianPetraStadlerEtAl17,
%	AlexanderianPetraStadlerEtAl14, CrestelAlexanderianStadlerEtAl17,
%	PetraMartinStadlerEtAl14, IsaacPetraStadlerEtAl15, SaibabaAlexanderianIpsen17, ChenVillaGhattas17, AlexanderianSaibaba18, ChenVillaGhattas19, ChenGhattas19a, ChenWuChenEtAl19a, ChenGhattas20, OLeary-RoseberryVillaChenEtAl20} explore and prove the fast decay of the eigenvalues for some model problems and numerically demonstrated for many others.
%Exploiting this property, efficient methods have been developed to evaluate D-optimal criterion for infinite-dimensional Bayesian linear problems \cite{SaibabaKitanidis15, SaibabaAlexanderianIpsen17}. 
%In \cite{AlexanderianSaibaba18}, the authors proposed a gradient-based optimization method for D-optimal experimental design. 

Here, we focus on \emph{goal-oriented} optimal experimental design (GOOED) for large-scale Bayesian inverse problems, in the context of optimal sensor placement. That is, we seek optimal sensor locations that maximize the information gained from the sensors, not about the model parameters, but (of greater practical interest) for a posterior model-predictive goal. 
In particular, we consider linear parameter-to-observable (PtO) maps governed by expensive models (e.g., PDEs) with high-dimensional uncertain parameters (e.g., infinite-dimensional before discretization). In \cite{AttiaAlexanderianSaibaba18}, a gradient-based optimization method is developed to solve the linear GOOED problem to find the optimal sensor locations. However, in each of the possibly very large number of iterations, many model evaluations have to be performed, which makes the algorithm prohibitive if each model evaluation (e.g., solving PDEs) is very expensive.

% that the quantity of interest $QoI$ that we want to identify the uncertainty of is not the model parameter itself, but some functions of the reconstructed parameter, for example, some prediction quantities. In this situation, it might be more desirable to collect data that can minimize the uncertainty of the QoI rather than the model parameter. On the other hand, the model parameter are often high-dimensional, or even infinite dimensional functions. The resulting data misfit Hessian is also very high dimensional after discretization, leading to extremely expensive computational cost to solve the optimization problem, while the dimension of the QoI is usually low-dimensional, much smaller than the parameter dimension. Thus incorporating the OoI may have computational benefits. These motivate us to develop  computational frameworks and efficient algorithms for GOOED. In \cite{AttiaAlexanderianSaibaba18}, the authors develop a computational framework for GOOED for large-scale Bayesian linear inverse problem, using a gradient-based method to solve the optimization problem, with each optimal criterion evaluation and its gradient involves many expensive forward model solves.

\textbf{Contributions}. We propose a fast and scalable method for high-dimensional and Bayesian GOOED problems governed by large-scale, expensive-to-solve models. To overcome the curse-of-dimensionality with respect to both parameter and data dimensions, we propose a new computational framework for the EIG with Cholesky factorization and exploit the intrinsic low-dimensionality of the data- and parameter-informed operators. The low-rank properties are revealed by Jacobians and Hessians of the PtO map,
% and higher-order derivatives,
as has been done for 
 model reduction for sampling and deep learning \cite{BashirWillcoxGhattasEtAl08,
	ChenGhattas19a, AlgerChenGhattas20, OLeary-RoseberryVillaChenEtAl22}, Bayesian inference
\cite{Bui-ThanhBursteddeGhattasEtAl12, Bui-ThanhGhattasMartinEtAl13, 
	ChenVillaGhattas17, ChenWuChenEtAl19, ChenWuGhattas20, ChenGhattas20}, optimization
under uncertainty \cite{AlexanderianPetraStadlerEtAl17,ChenVillaGhattas19, ChenHabermanGhattas21}, and BOED
\cite{AlexanderianPetraStadlerEtAl14, AlexanderianPetraStadlerEtAl16,
	CrestelAlexanderianStadlerEtAl17, SaibabaAlexanderianIpsen17, AttiaAlexanderianSaibaba18, WuChenGhattas20}. We use a randomized algorithm for the low-rank approximations, which require only a small and dimension-independent number of large-scale model evaluations and we provide a detailed error analysis for the approximated EIG. Moreover, with the proposed EIG framework, we are able to adopt an efficient offline-online decomposition to solve the optimization problem, where in the offline stage the model-constrained low-rank approximations are performed just once, while in the online stage the design optimization is performed free of model evaluations. Furthermore, for the design optimization, we use a swapping greedy algorithm that first constructs an initial set of sensors using leverage scores, and then swaps the chosen sensors with other candidates until certain convergence criteria are met. Finally, we demonstrate the efficiency, accuracy, and dimension independence (with respect to both data and parameters) of the proposed algorithm for a contaminant transport inverse problem with infinite-dimensional parameter field.

We present background on BOED in \cref{sec:background}, propose our computational framework for GOOED in 
\cref{sec:GOOED}, and report results on experiments in \cref{sec:experiments}. 

%----------------------------------------------------------------------------------------
%	Background
%----------------------------------------------------------------------------------------
\section{Background}
\label{sec:background}

%----------------------------------------------------------------------------------------
%	Bayesian inverse problem
%----------------------------------------------------------------------------------------
\subsection{Linear Bayesian inverse problem}
\label{sec:linearproblem}
We consider a general linear model 
\begin{equation}
	\label{eq:model}
	\obs = \mathbf{F}\bipar +  \boldsymbol{\epsilon},
\end{equation}
where $\obs \in \R^{d_y}$ is a $d_y$-dimensional observational data vector corrupted by additive Gaussian noise $\boldsymbol{\epsilon} \in \mathcal{N}(\mathbf{0}, \Gnoise)$ with zero mean and covariance $\Gnoise \in \R^{d_y\times d_y}$, $\bipar \in \R^{d_m}$ is a $d_m$-dimensional uncertain parameter vector,  and $\mathbf{F}: \R^{d_m} \mapsto \R^{d_y}$ is a linear PtO map. As a specific case, $\bipar$ is a discretization (e.g., by finite element method) of an infinite-dimensional parameter field in a model described by PDEs, while $\mathbf{F}$ is implicitly given by solving the PDE model.  In this case, the parameter dimension is typically very high, $O(10^6-10^9)$ for practical applications. 

%where $ \boldsymbol{\epsilon}$ is the observation noise. Specifically, we consider data $\obs \in \mathcal{Y}=\R^{n_y}$ is obtained at $n_y$ observation points with an additive Gaussian noise $ \mathbf{\epsilon} \sim \mathcal{N}(0,\Gnoise)$, where $\Gnoise \in \R^{n_y \times n_y}$ is a covariance matrix. $m$ is an  infinite-dimensional parameter field living in a Hilbert space $\mathcal{M}$ defined in a physical domain $\mathcal{D} \subset \R^{n_x}$. $\mathcal{F}$ is a linear parameter-to-observable (PtO) map. Under the assumption of additive Gaussian noise and a Gaussian prior, we have illustrated the infinite-dimensional framework, finite-dimensional discretization of the problem in details in [cite our OED paper]. \\
%We use a finite element discretization to approximate parameter $m$ as $\bipar \in \R^n$ as discussed in [cite our paper] with Euclidean inner product weighted by the finite element mass matrix $\mathbf{M}$. We denote $\mathbf{F} \in \R^{n_y \times n}$ as the discretized linear PtO map with its adjoint $\mathbf{F}^* = \mathbf{M}^{-1}\mathbf{F}^T$ explained in [cite Omar's paper].  \\

We assume a Gaussian prior $\bipar \sim \mathcal{N}(\bipar_{\text{pr}}, \bCpr)$ with mean $\bipar_{\text{pr}}$ and covariance $\bCpr$ for the parameter $\bipar $ with density  
\begin{equation}
\pi_{\text{pr}} (\bipar) \propto \exp\left(-\frac{1}{2} ||\bipar - \bipar_{\text{pr}}||^2_{\bCpr^{-1}} \right),
\end{equation}
where $||\bipar - \bipar_{\text{pr}}||^2_{\bCpr^{-1}} := (\bipar - \bipar_{\text{pr}})^T\bCpr^{-1}(\bipar - \bipar_{\text{pr}})$.
Then by Bayes' rule the posterior density of $\bipar$ satisfies
\begin{equation}
	\pi_{\text{post}} (\bipar | \obs) \propto \pilike(\obs|\bipar)  \pi_{\text{pr}}(\bipar).
\end{equation}
Here $\pilike(\obs|\bipar)$ is the likelihood function that satisfies 
\begin{eqnarray}
\pilike(y|\bipar) \propto \exp\left( - \Phi(\mathbf{m}, \obs)\right)
\end{eqnarray}
under Gaussian noise $\boldsymbol{\epsilon} \in \mathcal{N}(\mathbf{0}, \Gnoise)$ , where  the potential 
\begin{equation}
\Phi(\mathbf{m}, \obs) := \frac{1}{2} ||\mathbf{F}\mathbf{m} - \obs||^2_{\Gnoise^{-1}}.
\end{equation}
Under the assumption of Gaussian prior and Gaussian noise,  the posterior of $\bipar$ is also Gaussian $\mathcal{N}(\bipar_{\text{map}},\bpostcov)$ with mean 
$
\bipar_{\text{post}} = \bpostcov(\mathbf{F}^* \Gnoise^{-1} \obs + \bCpr^{-1} \bipar_{\text{pr}})
$
and covariance 
$
\bpostcov = (\bHmisfit  + \bCpr^{-1} )^{-1}
$, where 
\begin{equation}
\bHmisfit = \mathbf{F}^* \Gnoise^{-1} \mathbf{F}
\end{equation}
is the (data-misfit) Hessian of the potential $\Phi(\mathbf{m}, \obs)$, and $\mathbf{F}^*$ is the adjoint of $\mathbf{F}$, e.g., by solving the adjoint PDE model.
%----------------------------------------------------------------------------------------
%	Bayesian OED
%----------------------------------------------------------------------------------------
\subsection{Bayesian optimal experimental design}
\subsubsection{Expected information gain}
% Different criteria have been used to measure the information gain from the observation data, such as A-optimal and D-optimal criterion, which compute the trace and determinant of the covariance of the posterior. 
The expected information gain (EIG) is defined as the expected (with respect to data) Kullback-Leibler (KL) divergence between the posterior and the prior distributions, 
\begin{equation}\label{eq:EIG}
\Psi :=  \ave_\obs[D_{\text{KL}}(\pi_{\text{post}}(\cdot | \obs) \|\pi_{\text{pr}} )], 
\end{equation}
where the KL divergence is defined as  
\begin{equation}\label{eq:KL}
	D_{\text{KL}}(\pi_{\text{post}}\|\pi_{\text{pr}} ) := \int \ln\left(
	\frac{d\pi_{\text{post}}}{d\pi_{\text{pr}}}\right) d\pi_{\text{post}}.
\end{equation}
For a Bayesian linear inverse problem as formulated in \cref{sec:linearproblem}, the EIG $\Psi$ admits the closed form \cite{AlexanderianGloorGhattas16}
\begin{equation}\label{eq:EIGdiscrete}
\Psi = \frac{1}{2}\logdet{\mathbf{I}_m + \bpHmisfit},
\end{equation}
where $\mathbf{I}_m$ is an identity matrix of size $d_m\times d_m$, and $\bpHmisfit := \bCpr^{\frac{1}{2}}\bHmisfit \bCpr^{\frac{1}{2}}$ is the \textit{prior-preconditioned Hessian} that includes both data and prior information.

%----------------------------------------------------------------------------------------
%	OED for sensor placement
%----------------------------------------------------------------------------------------
\subsubsection{BOED for sensor placement}
\label{sec:sensor}
We consider an optimal sensor placement problem. Assume we have a collection of $d$ candidate sensors $\{ s_i\}^d_{i=1}$. We need to choose a much smaller number $r < d$ of sensors (due to a limited budget or physical constraints) at which data are collected. The OED problem seeks to find the best sensor combination from the candidates. We use a Boolean design matrix $W \in \mathcal{W} \subset \R^{r\times d} $ to represent sensor placement such that $W_{ij} = 1$ if the $i$-th sensor is placed at the $j$-th candidate location, i.e.,
\begin{equation}\label{eq:W}
W_{ij} \in \{0,1 \},\; 
\sum^d_{j=1} W_{ij}=1, \; \sum^r_{i=1} W_{ij} \in  \{0,1 \}.
\end{equation}

We assume that the observational noise for the $d$ candidate sensors is uncorrelated, with covariance  \begin{equation}\label{eq:Gamma_n_d}
\Gnoise^d = \text{diag}(\sigma_1^2,\dots, \sigma_d^2).
\end{equation}
As a result, for any design $W$ with the covariance for the observation noise $\boldsymbol{\epsilon}$ as $\Gnoise(W) = W \Gnoise^d W^T$, we have 
\begin{equation}
\Gnoise^{-1}(W) = W (\Gnoise^d)^{-1}W^T.
\end{equation}
Denoting by $\mathbf{F}_d$ the PtO map using all $d$ candidate sensors,  we have the design-specific PtO map 
\begin{equation}\label{eq:FW}
\mathbf{F}(W) = W \mathbf{F}_d,
\end{equation}
with its adjoint $\mathbf{F}^* = \mathbf{F}_d^* W^T$. 
We can now state the OED problem as: find an optimal design $W \in \mathcal{W}$ such that 
\begin{equation}
W = \argmax_{W \in  \mathcal{W} }\Psi(W).
\end{equation}
%----------------------------------------------------------------------------------------
%	GOOED
%----------------------------------------------------------------------------------------
\section{Goal-oriented optimal experimental design}
\label{sec:GOOED}

The classical OED problem seeks a design that maximizes the information gain for the parameter vector $\bipar$. In this work, we consider a goal-oriented optimal experimental design (GOOED) problem that maximizes the information gain of a predicted \textit{quantity of interest} (QoI) $\qoi \in \R^p$, which is assumed to be a linear function of the parameter $\bipar$,
\begin{equation}\label{eq:QoI}
	\qoi = \bP \bipar,
\end{equation}
where $\bP: \R^{d_m} \mapsto \R^{d_\rho}$ is a linear map that typically involves model evaluation (e.g., solving PDEs). 
% We consider the case where the dimension of $\qoi$ is much lower than the parameter dimension $n$ motivated by many practical applications (e.g., the temperature field over a subregion of the entire domain). 
Due to linearity, the prior distribution of $\qoi$ is Gaussian $\mathcal{N} (\qoi_{\text{pr}},\mathbf{\Sigma}_{\text{pr}})$ with mean $
	\qoi_{\text{pr}} = \bP \bipar_{\text{pr}}$ and covariance $ \mathbf{\Sigma}_{\text{pr}}= \bP \bCpr \bP^*$,
where $\bP^*$ is the adjoint of $\bP$. Moreover, the posterior distribution of $\qoi$ is also Gaussian $\mathcal{N} (\qoi_{\text{post}},\mathbf{\Sigma}_{\text{post}})$ with mean $
\qoi_{\text{post}} = \bP \bipar_{\text{post}}$ and covariance $ \mathbf{\Sigma}_{\text{post}}= \bP \bpostcov \bP^*$.

%The full Bayesian solution to the goal-oriented inverse problem is the posterior of the QoI, i.e., $\pi_{\qoi|\obs}$, is also Gaussian with mean and covariance given by
%\begin{equation}
%	\qoi_{\text{post}} = \bP \bipar_{\text{post}}, \mathbf{\Sigma}_{\text{post}}= \bP \bpostcov \bP^*.
%\end{equation}
%----------------------------------------------------------------------------------------
%	EIG for GOOED
%----------------------------------------------------------------------------------------
\subsection{Expected information gain for GOOED}
To construct an expression for EIG for GOOED, we first introduce \cref{prop:qoi} \cite{SpantiniCuiWillcoxEtAl17}, which relates the observational data $\obs$ and the QoI $\qoi$.
\begin{proposition}\label{prop:qoi}
Model \eqref{eq:model} and QoI \eqref{eq:QoI} lead to
	\begin{equation}\label{eq:model-rho}
		\obs = \bF \bP_{\dagger} \qoi + \boldsymbol{\eta},
	\end{equation}
	where $\bP_{\dagger}:=\mathbf{\Gamma}_{\text{pr}} \bP^* \mathbf{\Sigma}_{\text{pr}}^{-1}$, and $\boldsymbol{\eta} \sim \mathcal{N}(\mathbf{0},\mathbf{\Gamma}_{\eta})$ with 
	\begin{equation}\label{eq:GammaEta}
		\mathbf{\Gamma}_{\eta} := \Gnoise+\bF(\bCpr-\bCpr\bP^*\mathbf{\Sigma}_{\text{pr}}^{-1}\bP\bCpr)\bF^*,
	\end{equation}
	or equivalently $\mathbf{\Gamma}_{\eta} =\mathbb{C}ov [\boldsymbol{\epsilon}]+\mathbb{C}ov[\bF(\mathbf{I}_m-\bP_{\dagger}\bP)\bipar]$, with $\mathbb{C}ov$ as covariance. Moreover, $\qoi$ and $\boldsymbol{\eta}$ are independent.
\end{proposition}
%\begin{remark}
%	\label{rem:1}
%	Notice that $\mathbf{\Gamma}_{\eta} =\mathbb{V}ar [\boldsymbol{\epsilon}]+\mathbb{V}ar[\bF(I-\bP_{\dagger}\bP)\bipar]$.
%\end{remark}
Thus, the EIG for $\qoi$ can be obtained analogously to \eqref{eq:EIGdiscrete},
\begin{equation}\label{eq:EIGz}
	\Psi^{\rho}(W) = \frac{1}{2}\logdet{\mathbf{I}_\rho + \bpHmisfit^{\rho}(W)},
\end{equation}
where $\mathbf{I}_\rho$ is an identity matrix of size $d_\rho \times d_\rho$, and $\bpHmisfit^{\rho}(W) = \mathbf{\Sigma}_{\text{pr}}^{\frac{1}{2}}\bHmisfit^{\rho}(W) \mathbf{\Sigma}_{\text{pr}}^{\frac{1}{2}}$, with $\bHmisfit^{\rho}(W)$ given by
\begin{equation}
	\bHmisfit^{\rho} (W) = (\bF(W)\bP_{\dagger})^*\mathbf{\Gamma}_{\eta}^{-1}(W) \bF(W)\bP_{\dagger}.
\end{equation}

\subsection{Offline-online decomposition for EIG $\Psi^{\rho}$}
The EIG $\Psi^\rho(W)$ depends on $W$ through $\mathbf{F}(W)$ given in \eqref{eq:FW}, which involves expensive model evaluations (e.g., PDE solutions). Since  $\Psi^\rho(W)$ must be evaluated repeatedly in the course of maximizing EIG, these repeated model evaluations would be prohibitive. To circumvent this problem, % To eliminate the dependence of $\Psi^\rho(W)$ on model evaluations for each given  $W$
% which plays a key role in our fast computational framework, 
we propose an \emph{offline-online decomposition} scheme, where model-constrained computation of  quantities that are independent of $W$ is performed offline a single time, and the online design optimization is free of any model evaluations. The key result permitting this decomposition is given in the following theorem.
\begin{theorem}\label{eq:OffOn}
For each design $W \in \mathcal{W}$, the goal-oriented EIG $\Psi^\rho(W)$ given in \eqref{eq:EIGz} can be computed as 
\begin{equation}\label{eq:PsiRhoW}
\Psi^\rho(W) =  \frac{1}{2}\logdet{\mathbf{I}_r +L^T W  \mathbf{H}^{\rho}_d W^TL},
\end{equation}
where $\mathbf{I}_r$ is an identity matrix of size $r\times r$, $\mathbf{H}^{\rho}_d $ is given by 
\begin{equation}\label{eq:Hdrho}
\mathbf{H}^{\rho}_d  := \bF_d\bCpr\bP^*\mathbf{\Sigma}_{\text{pr}}^{-1}\bP\bCpr\bF_d^*,
\end{equation}
and
$L$ is given by the Cholesky factorization $\mathbf{\Gamma}_{\eta}^{-1} = L L^T$.
\end{theorem}
\begin{proof}
	To start with, we introduce the Weinstein-Aronszajn identity in \cref{prop:WA} which is proven in \cite{Pozrikidis14}. \begin{proposition}\label{prop:WA}
		Let $A$ and $B$ be matrices of size $m \times n$ and $n \times m$ respectively, then
		\begin{equation}\label{eq:WA}
			\det (\mathbf{I}_{n\times n} + BA) = \det (\mathbf{I}_{m\times m} + AB). 
		\end{equation}
	\end{proposition}	    
	Considering the design problem defined in \cref{sec:sensor}, for each design with design matrix $W$, we have 
	\begin{equation}
		\mathbf{F}(W) = W \mathbf{F}_d, \text{and } \Gnoise(W) = W \Gnoise^dW^T.
	\end{equation}
	We can then reformulate $\Psi^{\rho}$ with the definition in \cref{eq:EIGz} as
	\begin{equation}\label{eq:EIGz2}
		\begin{split}
			{\Psi}^{\rho}(W)&  = \frac{1}{2}\log\det(\mathbf{I} + \bpHmisfit^{\rho})\\
			& = \frac{1}{2}\log\det(\mathbf{I} +\mathbf{\Sigma}_{\text{pr}}^{\frac{1}{2}} (W\bF_d\bP_{\dagger})^*\mathbf{\Gamma}_{\eta}^{-1}(W)(W\bF_d \bP_{\dagger})\mathbf{\Sigma}_{\text{pr}}^{\frac{1}{2}}).
		\end{split}
	\end{equation}
	where 
	\begin{equation}
		\label{eq:gammaeta}
		\begin{split}
			\mathbf{\Gamma}_{\eta}(W) & = \Gnoise(W)+\bF(W)(\bCpr-\bCpr\bP^*\mathbf{\Sigma}_{\text{pr}}^{-1}\bP\bCpr)\bF^*(W)\\
			& = W \Gnoise^dW^T+W\bF_d(\bCpr-\bCpr\bP^*\mathbf{\Sigma}_{\text{pr}}^{-1}\bP\bCpr)\bF_d^*W^T\\
			& = W (\Gnoise^d+\bF_d(\bCpr-\bCpr\bP^*\mathbf{\Sigma}_{\text{pr}}^{-1}\bP\bCpr)\bF_d^*)W^T\\
			& = W (\Gnoise^d+\underbrace{\bF_d\bCpr\bF_d^*}_{:= \bH_d \in \R^{d\times d}}-\underbrace{\bF_d\bCpr\bP^*\mathbf{\Sigma}_{\text{pr}}^{-1}\bP\bCpr\bF_d^*}_{:=\bH_d^{\rho}\in \R^{d\times d}})W^T\\
			& = W (\Gnoise^d+\underbrace{\bH_d-\bH_d^{\rho}}_{:=\Delta \bH_d})W^T\\
			& = W (\Gnoise^d+\Delta \bH_d)W^T.
		\end{split}
	\end{equation}

    To this end, we have 
	\begin{equation}
		\begin{split}
			{\Psi}(W)
			& = \frac{1}{2}\logdet{\mathbf{I} +\mathbf{\Sigma}_{\text{pr}}^{\frac{1}{2}} (W\bF_d\bP_{\dagger})^*\mathbf{\Gamma}_{\eta}^{-1}(W) (W\bF_d \bP_{\dagger})\mathbf{\Sigma}_{\text{pr}}^{\frac{1}{2}}}\\
			& =  \frac{1}{2}\logdet{\mathbf{I} +\underbrace{\mathbf{\Sigma}_{\text{pr}}^{\frac{1}{2}} (W\bF_d\bP_{\dagger})^* L}_{A}\underbrace{L^T (W\bF_d \bP_{\dagger})\mathbf{\Sigma}_{\text{pr}}^{\frac{1}{2}}}_{B}}\\
			& = \frac{1}{2}\logdet{\mathbf{I} +\underbrace{L^T (W\bF_d \bP_{\dagger})\mathbf{\Sigma}_{\text{pr}}^{\frac{1}{2}}}_{B}\underbrace{\mathbf{\Sigma}_{\text{pr}}^{\frac{1}{2}} (W\bF_d\bP_{\dagger})^* L}_{A}}\\
			& = \frac{1}{2}\logdet{\mathbf{I} +L^T (W\bF_d \bP_{\dagger})\mathbf{\Sigma}_{\text{pr}} (W\bF_d\bP_{\dagger})^* L}\\
			& =  \frac{1}{2}\logdet{\mathbf{I} +L^T W\bF_d \mathbf{\Gamma}_{\text{pr}} \bP^* \mathbf{\Sigma}_{\text{pr}}^{-1}\mathbf{\Sigma}_{\text{pr}} \mathbf{\Sigma}_{\text{pr}}^{-1} \bP \mathbf{\Gamma}_{\text{pr}}\bF^*_dW^TL}\\
			& = \frac{1}{2}\logdet{\mathbf{I} +L^T W \bF_d \mathbf{\Gamma}_{\text{pr}} \bP^* \mathbf{\Sigma}_{\text{pr}}^{-1} \bP \mathbf{\Gamma}_{\text{pr}}\bF^*_d W^T L}\\
			& = \frac{1}{2}\logdet{\mathbf{I} +L^T W  \mathbf{H}^{\rho}_d W^TL},
		\end{split}
	\end{equation}
    where we use the Cholesky decomposition $\mathbf{\Gamma}_{\eta}^{-1} = L L^T$ in the second equality, \cref{prop:WA} in the third, definition of $\bP_{\dagger}$ from \eqref{eq:model-rho} in the fifth, and definition of $\mathbf{H}^{\rho}_d$ from \eqref{eq:Hdrho} in the last.
	\end{proof}

Note that $\mathbf{\Gamma}_{\eta}$ defined in \eqref{eq:GammaEta} can be equivalently written as 
\begin{equation}\label{eq:GammaW}
\mathbf{\Gamma}_{\eta}(W) = W \left(\Gnoise^d +  \mathbf{H}_d-  \mathbf{H}^{\rho}_d\right) W^T,
\end{equation}
where $\mathbf{H}_d := \mathbf{F}_d\bCpr\mathbf{F}_d^*$.
Hence evaluation of $\Psi^\rho(W)$ can be decomposed as follows: (1) construct the model-constrained matrices $\mathbf{H}_d$ and $\mathbf{H}^{\rho}_d$ offline just once; and (2) for each $W$ in the online optimization process, assemble a small ($r\times r$) matrix $\mathbf{\Gamma}_{\eta}(W) $ by \eqref{eq:GammaW}, compute a Cholesky factorization $\mathbf{\Gamma}_{\eta}^{-1} = L L^T$, and assemble $\Psi^\rho(W)$ by \eqref{eq:PsiRhoW}, which are all free of the expensive model evaluations. 

Note that $\bH_d \in \R^{d\times d}$ and $\bH_d^{\rho} \in \R^{d \times d}$ are large matrices when we have a large number of candidate sensors $d \gg 1$. Moreover, their construction involves expensive model evaluations when the parameters are high-dimensional, $d_m \gg 1$, e.g., by solving PDEs. Therefore, it is computationally not practical to directly compute and store these matrices. Fortunately, the intrinsic ill-posedness of high-dimensional Bayesian inverse problems---data inform only a low-dimensional subspace of parameter space, e.g., \cite{IsaacPetraStadlerEtAl15, Bui-ThanhBursteddeGhattasEtAl12, PetraMartinStadlerEtAl14, FlathWilcoxAkcelikEtAl11, AmbartsumyanBoukaramBui-ThanhEtAl20}---suggests that these matrices are likely low rank or exhibit rapid spectral decay. We exploit this property and construct low-rank approximations of $\bH_d^{\rho}$ and $\bH_d - \bH_d^{\rho}$ in the next section.

%However, when parameter dimension $n$ and the number of candidate sensor $d$, which we can also treat as data dimension, are large, $\bF_d$ involves some large-scale discrete differential operator, which is computational prohibitive for the evaluation of $\bH_d$ and $\bH_d^{\rho}$. It is not desirable to evaluate and store full matrices of $\bH_d$ and $\bH_d^{\rho}$. Due to the common ill-postness of the high-dimensional inverse problem, the data only effectively infer a low-dimensional subspace of whole parameter and data space. This indicates a possible fast decay of eigenvalues of $\bH_d$ and $\bH_d^{\rho}$ that we are going to explore to make low-rank approximations for them in the next section.

%----------------------------------------------------------------------------------------
%	Low-rank approximation
%----------------------------------------------------------------------------------------
\subsection{Low-rank approximation}
\label{sec:svd}
Let $\Delta \bH_d := \bH_d - \bH_d^{\rho}$, where $\bH_d^{\rho}$ and $\bH_d$ are given in \eqref{eq:Hdrho} and \eqref{eq:GammaW} and integrate data, parameter, and QoI information. Noting that $\bH_d^{\rho}$ and $\Delta \bH_d$ are both symmetric, we compute their low-rank approximation for given tolerances $\epsilon_\zeta, \epsilon_\lambda > 0$ as 
	\begin{equation}\label{eq:svd}
	\hat{\bH}_d^{\rho} = U_k Z_k U_k^T \quad \text{and} \quad \Delta\hat{\bH}_d = V_l\Lambda_l  V_l^T,
\end{equation}
where $(U_k, Z_k)$ represent the $k$ dominant eigenpairs of $\bH_d^{\rho}$ with $Z_k = \text{diag} (\zeta_1,\dots,\zeta_k)$ such that
\begin{equation}
\zeta_1 \geq \zeta_2 \geq \cdots \geq \zeta_k \geq \epsilon_\zeta \geq \zeta_{k+1}\cdots \geq \zeta_d;
\end{equation}
and $(V_l, \Lambda_l)$ represent the $l$ dominant eigenpairs of $\Delta \bH_d$ with $\Lambda_l = \text{diag} (\lambda_1,\dots,\lambda_l)$ such that 
 \begin{equation}
\lambda_1 \geq \lambda_2 \geq \cdots \geq \lambda_l\geq \epsilon_\lambda \geq \lambda_{l+1} \geq \cdots \geq \lambda_d.
 \end{equation}
For the low-rank approximation, we employ a randomized SVD algorithm \cite{HalkoMartinssonTropp11}, which requires only $O(k)$ and $O(l)$ model evaluations, respectively. In practice, $k, l \ll d$. More details on the algorithm applied to the example problem in \cref{sec:experiments} can be found in Appendix A.

With $\hat{\mathbf{\Gamma}}_{\eta} (W) := W \left(\Gnoise^d +  \Delta\hat{\bH}_d\right) W^T$ as an approximation of $\mathbf{\Gamma}_{\eta}(W)$ in \eqref{eq:GammaEta}, we compute the Cholesky factorization $\hat{\mathbf{\Gamma}}_{\eta}^{-1} = \hat{L} \hat{L}^T$. Then we can define an approximate EIG as 
\begin{equation}\label{eq:Psihat}
	\hat{ \Psi }^{\rho}(W) := \frac{1}{2}\logdet{\mathbf{I}_{r}+ \hat{L}^T W \hat{\bH}^{\rho}_d W^T\hat{L}}.
\end{equation}	
The following theorem quantifies the approximation error.
%Now we can present our main theorem of a efficient low-rank approximation of the EIG $\Psi^{\rho}$ in (\cref{eq:EIGz3}) with proof in Appendix.
\begin{theorem}\label{thm:low-rank-EIG}
	For any design $W \in \mathcal{W}$, the error for the goal-oriented EIG $\Psi^\rho(W)$ in \eqref{eq:PsiRhoW} by its approximation $\hat{ \Psi }^{\rho}(W) $ in \eqref{eq:Psihat} can be bounded by 
	\begin{equation}\label{eq:boundPsi}
		\begin{split}
		| \Psi^\rho(W) -\hat{ \Psi }^\rho(W)| & \leq  \frac{1}{2} \sum_{i=k+1}^d\log(1 + \zeta_i/\sigma^2_{\text{min}}) \\
		& + \frac{1}{2}\sum_{i=l+1}^k \log(1+ \lambda_i \zeta_1/\sigma_{\text{min}}^4),			
		\end{split}
	\end{equation}
	where $\sigma^2_{\text{min}}:=\min(\sigma^2_1, \dots, \sigma_d^2)$ as defined in \cref{eq:Gamma_n_d}.
\end{theorem}

\begin{proof}
 	We first introduce necessary properties that are proven in \cite{WuChenGhattas20} for \cref{prop:eig},  \cite{AlexanderianSaibaba18} for \cref{prop:det} and \cite{Yue20} for 
	\cref{prop:rearrangement}.
	\begin{proposition}\label{prop:eig}
		Let $A$ and $B$ be matrices of size $m \times n$ and $n \times m$ respectively, then $AB$ and 
		$BA$ have the same non-zero eigenvalues.
	\end{proposition}
	\begin{proposition}\label{prop:det}
		Let $A,B \in \mathbb{C}^{n \times n}$ be Hermitian positive semidefinite with $A \geq B$  (i.e., $A-B$ is Hermitian positive semidefinite), then
		\begin{equation}
			0 \leq \log \det (I+A) -\log \det (I+B)  \leq \log \det (I+A-B).
		\end{equation}
	\end{proposition}

	\begin{proposition}\label{prop:rearrangement}
		Let $f: \R_+ \to \R$ be a continuous function that is differentiable on $\R_{+}$ (with $x \geq 0$ for $x \in \R_+$). If the function $x \mapsto x f'(x)$ is monotonically increasing on $\R_{+}$. Then for any matrices $A,B \in \R^{n\times m}$, it holds that 
		\begin{equation}
			\sum^n_{i=1} f(\upsilon_i(AB^T)) \leq  \sum^n_{i=1} f(\upsilon_i(A)\upsilon_i(B))
		\end{equation}
		where $\upsilon_i(\cdot)$ denotes the singular values of matrices sorted in non-increasing order. 
	\end{proposition}
	\begin{lemma}\label{lem:rearrangement}
		Let $A \in \R^{n\times m}, B \in \R^{m \times m}$, $A^TA$ and $B$ are Hermitian positive semidefinite, then 
		\begin{equation}
			\log\det (I + A B A^T) \leq \sum^m_{i=1} \log(1+\upsilon_i(A^T A)\upsilon_i(B)).
		\end{equation}
	\end{lemma}
	\begin{proof}
		Since $\log\det (I + A B A^T) = \sum^n_{i=1} \log(1+\upsilon_i(ABA^T))=\sum^n_{i=1} \log(1+\upsilon^2_i(AB^{1/2}))$, let $f(x)=\log(1+x^2)$, which satisfies 
		\cref{prop:rearrangement}, we have 
		\begin{equation}
			\sum^n_{i=1} \log(1+\upsilon^2_i(AB^{1/2})) \leq \sum^n_{i=1} \log(1+\upsilon^2_i(A)\upsilon^2_i(B^{1/2}) )= \sum^m_{i=1} \log(1+\upsilon_i(A^TA)\upsilon_i(B)).
		\end{equation}
	\end{proof}
   
		Denote the eigenvalue decompositions of $\bH^{\rho}_d$ and $\Delta \bH_d$ as 
		\begin{equation}
			\bH_d^{\rho} = U_k Z_k U_k^T + U_\perp Z_\perp U_\perp^T, \text{ and } \Delta \bH_d = V_l \Lambda_l V_l^T + V_\perp \Lambda_\perp V_\perp^T,
		\end{equation}
		where $(Z_k, U_k),(V_l, \Lambda_l)$ represent the dominant eigenpairs, and $(Z_\perp, U_\perp),(V_\perp, \Lambda_\perp)$ represent the remaining eigenpairs. By triangle inequality, we have
		\begin{equation}
			\begin{split}
				|\Psi^{\rho}(W) -\hat{ \Psi }^{\rho}(W)|  &= 
				| \frac{1}{2}\logdet{\mathbf{I}_{r\times r} +L^T W  \mathbf{H}^{\rho}_d W^TL} - \frac{1}{2}\logdet{\mathbf{I}_{r\times r}+ \hat{L}^T W \hat{\bH}^{\rho}_d W^T\hat{L}}|\\
				& \leq \underbrace{|\frac{1}{2}\logdet{\mathbf{I}_{r\times r} +L^T W  \mathbf{H}^{\rho}_d W^TL} -\frac{1}{2}\logdet{\mathbf{I}_{r\times r}+L^T W  \hat{\mathbf{H}}^{\rho}_d W^TL}|}_{(a)}\\
				&~~~~+\underbrace{|\frac{1}{2}\logdet{\mathbf{I}_{r\times r} +L^T W  \hat{\mathbf{H}}^{\rho}_d W^TL}- \frac{1}{2}\logdet{\mathbf{I}_{r\times r}+ \hat{L}^T W \hat{\bH}^{\rho}_d W^T\hat{L}}|}_{(b)}.
			\end{split}
		\end{equation}
		We first look at $(a)$. By \cref{prop:det} and note that $(\mathbf{H}^{\rho}_d-\hat{\mathbf{H}}^{\rho}_d) = U_\perp Z_\perp U_\perp^T$ is Hermitian positive semidefinite, we have 
		\begin{equation}
			\begin{split}
				(a)& \leq \frac{1}{2}\logdet{\mathbf{I}_{r\times r} +L^T W  \mathbf{H}^{\rho}_d W^TL-L^T W  \hat{\mathbf{H}}^{\rho}_d W^TL} \\
				& = \frac{1}{2}\logdet{\mathbf{I}_{r\times r} +L^T W  (\mathbf{H}^{\rho}_d-\hat{\mathbf{H}}^{\rho}_d) W^TL} \\
				& = \frac{1}{2}\logdet{\mathbf{I}_{r\times r} +L^T W   U_\perp Z_\perp U_\perp^T W^TL}.
			\end{split}
		\end{equation}
		Then applying \cref{prop:WA}, we have 
		\begin{equation}
			\label{eq:a}
			\begin{split}
				(a)&  = \frac{1}{2}\logdet{\mathbf{I}_{r\times r} +L^T W   U_\perp Z_\perp^{1/2}Z_\perp^{1/2} U_\perp^T W^TL} \\
				&  = \frac{1}{2}\logdet{\mathbf{I}_{(d-k)\times (d-k)} +Z_\perp^{1/2} U_\perp^T W^TLL^T W   U_\perp Z_\perp^{1/2}} \\
				& = \frac{1}{2}\logdet{\mathbf{I}_{(d-k)\times (d-k)} +Z_\perp^{1/2} U_\perp^T W^T( W(\Gnoise^d + \Delta\bH_d) W^T)^{-1}W   U_\perp Z_\perp^{1/2}}.
			\end{split}
		\end{equation}
		Applying \cref{lem:rearrangement}, let $A = Z_\perp^{1/2} U_\perp^T W^T, B = ( W(\Gnoise^d + \Delta\bH_d) W^T)^{-1}$, we have
		\begin{equation}
			\label{eq:a2}
			\begin{split}
				(a)& \leq \frac{1}{2} \sum_i\log(1 + \upsilon_i(W U_\perp Z_\perp^{1/2}Z_\perp^{1/2} U_\perp^T W^T)\upsilon_i(( W(\Gnoise^d + \Delta\bH_d) W^T)^{-1}))\\
				& = \frac{1}{2} \sum_i\log(1 + \upsilon_i(W U_\perp Z_\perp U_\perp^T W^T)\upsilon_i(( W(\Gnoise^d + \Delta\bH_d) W^T)^{-1})).
			\end{split}
		\end{equation}
		% ({\color{red} you can not change $v_i(S^{-1})$ to $v_i^{-1}(S)$ as they are ordered, say $S = \text{diag}(1/3, 1/2, 1)$, then $v_1(S^{-1}) = 3$ but $v_1^{-1}(S) = 1$, right?})
By \cref{prop:qoi}, $\Delta \bH_d = \mathbb{C}ov[\bF_d(I-\bP_{\dagger}\bP)\bipar]$, is a covariance matrix, thus is positive semidefinite. The smallest eigenvalue of $\Gnoise^d + \Delta\bH_d$ is greater than the smallest eigenvalue of $\Gnoise^d$. Hence $\upsilon_i( W(\Gnoise^d + \Delta\bH_d) W^T))\geq \sigma^2_{\text{min}}$, i.e., $\upsilon_i(( W(\Gnoise^d + \Delta\bH_d) W^T)^{-1}) \leq 1/\sigma_{\text{min}}^2$. Note that $\upsilon_i(W U_\perp Z_\perp U_\perp^T W^T) \leq \upsilon_i(U_\perp Z_\perp U_\perp^T) = \zeta_i$. Thus we have 
		\begin{equation}
			\label{eq:a3}
			(a) \leq \frac{1}{2} \sum_{i=k+1}^d\log(1 + \zeta_i/\sigma^2_{\text{min}}).
		\end{equation}
		Then we turn to second part $(b)$, with \cref{prop:WA} and \cref{prop:det}, we have
		\begin{equation}
			\begin{split}
				(b)& =|-\frac{1}{2}\logdet{\mathbf{I}_{r\times r} +L^T W  U_k Z_k U_k^T  W^TL}+ \frac{1}{2}\logdet{\mathbf{I}_{r\times r}+ \hat{L}^T W  U_k Z_k U_k^T W^T\hat{L}}|\\
				& = | -\frac{1}{2}\logdet{\mathbf{I}_{k\times k} +Z_k^{1/2} U_k^T  W^TLL^T W  U_k Z_k^{1/2}}+ \frac{1}{2}\logdet{\mathbf{I}_{k\times k}+ Z_k^{1/2} U_k^T W^T\hat{L}\hat{L}^T W  U_k Z_k^{1/2}}|\\
				& \leq \frac{1}{2}\logdet{\mathbf{I}_{k\times k} +  Z_k^{1/2} U_k^T W^T\hat{L}\hat{L}^T W U_k Z_k^{1/2} -Z_k^{1/2}U_k^T W^TLL^T W U_k Z_k^{1/2}}\\
				& = \frac{1}{2}\logdet{\mathbf{I}_{k\times k} + Z_k^{1/2}U_k^T W^T(\hat{L}\hat{L}^T-LL^T) W U_k Z_k^{1/2} }\\
				& =  \frac{1}{2}\logdet{\mathbf{I}_{k\times k} +Z_k^{1/2} U_k^T W^T\underbrace{(( W(\Gnoise^d + \Delta\hat{\bH}_d) W^T)^{-1}-( W(\Gnoise^d + \Delta\bH_d) W^T)^{-1})}_{(c)} W U_k Z_k^{1/2} }.
			\end{split}
		\end{equation}
		Note that $(A+B)^{-1} = A^{-1}-A^{-1}B(A+B)^{-1}$, let $A = W(\Gnoise^d + \Delta\hat{\bH}_d)W^T, B = W(\Delta\bH_d-\Delta\hat{\bH}_d)W^T=W V_\perp \Lambda_\perp V_\perp^T W^T$, we have 
		\begin{equation}
			\begin{split}
				& (A+B)^{-1} = ( W(\Gnoise^d + \Delta\bH_d)W^T)^{-1} \\
				& = (W(\Gnoise^d + \Delta\hat{\bH}_d)W^T)^{-1}-(W(\Gnoise^d + \Delta\hat{\bH}_d)W^T)^{-1}W V_\perp \Lambda_\perp V_\perp^T W^T( W(\Gnoise^d + \Delta\bH_d)W^T)^{-1}\\
				\Rightarrow &  (c)=(W(\Gnoise^d + \Delta\hat{\bH}_d)W^T)^{-1}W V_\perp \Lambda_\perp V_\perp^T W^T( W(\Gnoise^d + \Delta\bH_d)W^T)^{-1}
			\end{split}
		\end{equation}
		Then we can see that 
		\begin{equation}
			\begin{split}
				(b)&  \leq \frac{1}{2}\logdet{\mathbf{I}_{k\times k}+ Z_k^{1/2}U_k^T W^T  (W(\Gnoise^d + \Delta\hat{\bH}_d)W^T)^{-1}W V_\perp \Lambda_\perp V_\perp^T W^T( W(\Gnoise^d + \Delta\bH_d)W^T)^{-1}W U_k Z_k^{1/2} }\\
				& =\frac{1}{2}\logdet{\mathbf{I}_{(d-l)\times (d-l)}+ \Lambda_\perp^{1/2} V_\perp^T W^T( W(\Gnoise^d + \Delta\hat{\bH}_d)W^T)^{-1} W U_k Z_k^{1/2}Z_k^{1/2}U_k^T W^T ( W(\Gnoise^d + \Delta\bH_d)W^T)^{-1}W V_\perp\Lambda_\perp^{1/2} }\\
				& = \frac{1}{2}\logdet{\mathbf{I}_{(d-l)\times (d-l)}+ \Lambda_\perp^{1/2} V_\perp^T W^T( W(\Gnoise^d + \Delta\hat{\bH}_d)W^T)^{-1} W U_k Z_kU_k^T W^T ( W(\Gnoise^d + \Delta\bH_d)W^T)^{-1}W V_\perp\Lambda_\perp^{1/2} }.
			\end{split}
		\end{equation}
		Applying \cref{lem:rearrangement}, we have
		\begin{equation}\label{eq:b}
			\begin{split}
				(b)&  \leq  \frac{1}{2}\sum_i \log(1+ \upsilon_i(W V_\perp\Lambda_\perp V_\perp^T W^T) \upsilon_{i}(( W(\Gnoise^d + \Delta\hat{\bH}_d)W^T)^{-1} W U_k Z_kU_k^T W^T ( W(\Gnoise^d + \Delta\bH_d)W^T)^{-1} ))\\
				& \leq \frac{1}{2}\sum_{i=l+1}^k \log(1+ \lambda_i \zeta_1/\sigma_{\text{min}}^4),
			\end{split}
		\end{equation}
		where we have used 
		\begin{equation}
		    \begin{split}
    		&\upsilon_{i}(( W(\Gnoise^d + \Delta\hat{\bH}_d)W^T)^{-1} W U_k Z_kU_k^T W^T ( W(\Gnoise^d + \Delta\bH_d)W^T)^{-1} ) \\
    		& \leq v_1 (( W(\Gnoise^d + \Delta\hat{\bH}_d)W^T)^{-1}) v_1 (W U_k Z_kU_k^T W^T) v_1 (( W(\Gnoise^d + \Delta\bH_d)W^T)^{-1}) \leq \zeta_1/\sigma_{\text{min}}^4
		    \end{split}
		\end{equation}
		for $i \leq k$ in the last inequality. Note that it vanishes for $i > k$ as $Z_k$ has rank not larger than $k$.
		Combining \cref{eq:a3} and \cref{eq:b}, 
		\begin{equation}
			|\Psi^{\rho}(W) -\hat{ \Psi }^{\rho}(W)| \leq (a) + (b) \leq \frac{1}{2} \sum_{i=k+1}^d\log(1 + \zeta_i/\sigma^2_{\text{min}}) + \frac{1}{2}\sum_{i=l+1}^k \log(1+ \lambda_i \zeta_1/\sigma_{\text{min}}^4).
		\end{equation}
\end{proof}

We remark that with rapid decay of the eigenvalues $(\zeta_k)_{k\geq 1}$ of $\hat{\bH}_d^{\rho}$ and $(\lambda_l)_{l\geq 1}$ of $\Delta\hat{\bH}_d$, the error bound in \eqref{eq:boundPsi} becomes very small. Moreover, the decay rates are often independent of the (candidate) data dimension $d$ and the parameter dimension $d_m$, as demonstrated in \cref{sec:scalability}. This means that an arbitrarily-accurate EIG approximation can be constructed with a small number, $O(k+l)$, of model solves. 

\subsection{Swapping greedy optimization}
Once the low-rank approximations of $\bH_d^{\rho}$ and $\Delta \bH_d$ are constructed per  \eqref{eq:svd}, we obtain a fast method for evaluating the approximate EIG in \eqref{eq:Psihat}, with certified approximation error given by \cref{thm:low-rank-EIG}. We emphasize that this fast computation does not involve expensive model evaluations (e.g., large-scale PDE solves) for any given design $W$. We now turn to the (combinatorial) optimization problem of finding the optimal design matrix $W$,
% we deal with the optimization problem to find optimal design matrix $W$ such that 
\begin{equation}
	W = \argmax_{W\in\mathcal{W}} \hat{\Psi}^{\rho} (W).
\end{equation}
We next introduce a swapping greedy algorithm to solve this problem 
requiring only evaluation of $\hat{\Psi}^{\rho} (W)$. 
%Since evaluating the approximated EIG $\hat{\Psi}^{\rho}$ does not involve forward model solves, instead of a gradient-based method that needs the computation of its gradient in every optimization iteration, we can adopt a greedy algorithm as it only involves $\hat{\Psi}^{\rho}$ evaluations for a limited number of design choices. 
%%----------------------------------------------------------------------------------------
%%	Swapping greedy algorithm
%%----------------------------------------------------------------------------------------
%\subsubsection{Swapping greedy algorithm}

In contrast to classical greedy algorithms that sequentially find the optimal sensors one by one (or batch by batch) \cite{BianBuhmannKrauseEtAl17, JagalurMohanMarzouk20}, we extend a swapping greedy algorithm developed for BOED in \cite{WuChenGhattas20} to solve the GOOED problem. 
% It swaps a sensor at one time between a chosen sensor set represented by $W$ and the other candidate sensor set to maximize the approximate EIG $\hat{\Psi}^{\rho} (W)$ until it converges. 
Given a current sensor set, it swaps sensors with the remaining sensors to maximize the approximate EIG $\hat{\Psi}^{\rho} (W)$ until convergence. To initialize the chosen sensor set, we 
take advantage of the low-rank approximation $\hat{\bH}_d^{\rho}$ in \eqref{eq:svd}, which contains information from the data (through $\mathbf{F}_d$), parameter (through $\bCpr$),  and QoI (through $\bP$), as can be seen from \eqref{eq:Hdrho}. In particular, the most informative sensors can be revealed by the rows of $U_k$ with the largest norms, or the leverage scores of $\bH_d^{\rho}$ \cite{BoutsidisMahoneyDrineas08}. More specifically, given a budget of selecting $r$ sensors from $d$ candidate locations, we initialize the candidate set $S^0 = \{s_1, \dots, s_r\}$ such that $s_i$, $i = 1, \dots, r$, is the row index corresponding to the $i$-th largest row norm of $U_k$, i.e., 
\begin{equation}\label{eq:InitS}
	s_i = \argmax_{s \in S \setminus S_{i-1}} ||U_k(s,:)||_2, \quad i = 1, \dots, r,
\end{equation}
where $U_k(s,:)$ is the $s$-th row of $U_k$, $||\cdot||_2$ is the Euclidean norm, and the set $S_{i-1} = \{s_1, \dots, s_{i-1}\}$ for $i = 2, 3, \dots,$ and $S_0 = \emptyset$. 
 % {\color{red} check if this initialization scheme is correct, probably not because $\hat{\bH}_d^{\rho}$ could be rank 1 if the QoI is one dimensional.} 
 Then at each step of a loop for $t = 1, \dots, r$, we swap a sensor $s_t$ from the current chosen sensor set $S^{t-1}$ with one from the candidate set such that the approximate EIG $\hat{\Psi}^\rho(W)$ evaluated as in \eqref{eq:Psihat} can be maximized, i.e., we choose $s^*$ such that
 \begin{equation}\label{eq:UpdateS}
 	 s^* = \argmax_{s \in \{ s_t\} \cup (S \setminus S^{t-1})}\hat{ \Psi }^\rho(W_s),
 \end{equation}
where $W_s$ is the design matrix corresponding to the sensor choice $S^{t-1}  \setminus  \{ s_t\}\cup \{ s \}$. We repeat the loop until a convergence criterion is met, e.g., the chosen $S$ does not change or the difference of the approximate EIG is smaller than a given tolerance. We summarize the swapping algorithm algorithm in 
\cref{alg:swapping}.

\begin{algorithm}[H]
	\small
	\caption{A swapping greedy algorithm for GOOED}
	%\hspace*{\algorithmicindent} \textbf{Input}: $U, \Sigma$, design $W$. 
	\label{alg:swapping}
	\begin{algorithmic}[1]
		\STATE {\bfseries Input}: low-rank approximations \eqref{eq:svd}, a set $S = \{1, \dots, d\}$ of $d$ candidate sensors, a budget of $r$ sensors to be placed.
		\STATE {\bfseries Output}: the optimal sensor set $S^*$ with $r$ sensors. 
		\STATE Initialize $S^* =\{s_1,\dots,s_r\} \subset S$ according to \eqref{eq:InitS}. 
		\STATE Set $S^0 = \{ \emptyset\}$.
		\WHILE{$S^* \neq S^0$}
		\STATE $S^0 \leftarrow S^*$.
		\FOR{$t = 1,\dots,r$}
		\STATE Choose $s^*$ according to \eqref{eq:UpdateS}.
		\STATE Update $S^t \leftarrow (S^{t-1}  \setminus  \{ s_t\})\cup \{ s^*\}$.
		\ENDFOR
		\STATE Update $S^* \leftarrow S^r$.
		\ENDWHILE
		\STATE {\bfseries Output:} optimal sensor choice $S^*$.
	\end{algorithmic}
	%\hspace*{\algorithmicindent} \textbf{Output} : $\tilde{\psi}$
\end{algorithm}
%----------------------------------------------------------------------------------------
%	Computational complexity
%----------------------------------------------------------------------------------------
%\subsubsection{Computational complexity}
%We consider the number of forward model solves as the dominant computational cost. As we have discussed in the previous section, to solve this optimization problem, the low-rank approximation of $\bH_d^{\rho}$ and $\Delta \bH_d$ are the only two steps that requires forward model solves. For each $\bH_d^{\rho}$ application, it requires one pair of forward solve $\bF_d$ and its adjoint $\bF^*$ and two application of the prediction operator $\bP$ and its adjoint $\bP^*$. For each $\Delta \bH_d$ application, it needs one $\bF$ and application of $\bH_d^{\rho}$. 
%
%As we can compute two low-rank approximations together, in total, denote $c = \max(l,k)$, it involves $c+l$ forward solves $\bF$ and $c$ of its adjoint $\bF^*$, and $2c$ applications of $\bP$ and its adjoint $\bP^*$.

%----------------------------------------------------------------------------------------
%	Numerical results
%----------------------------------------------------------------------------------------
\section{Experiments}
\label{sec:experiments}
In this section, we present results of numerical experiments for GOOED governed by a linear dynamical PDE model with infinite-dimensional parameter field and varying numbers of candidate sensors. This problem features the key challenges of (1) expensive model evaluation and (2) high-dimensional parameters and data.
%The code is available at \href{https://github.com/cpempire/GOOED}{https://github.com/cpempire/GOOED}.
% a model of Bayesian linear inverse problem to demonstrate the effectiveness and efficiency of our proposed computational framework and greedy algorithm. 
%----------------------------------------------------------------------------------------
%	Model settings
%----------------------------------------------------------------------------------------
\subsection{Model settings}
\label{sec:model-setting}
We consider sensor placement for Bayesian inversion of a contaminant source with the goal of maximizing information gain for contaminant concentration on some building surfaces. The transport of the contaminant can be modeled by the time-dependent advection-diffusion equation with homogeneous Neumann boundary condition, 
%The forward problem: the PtO map $\G m := \mathcal{B}u (m)$ maps an initial condition $m \in \cL^2(\Omega)$  to pointwise spatiotemporal observations of the concentration field $u(x,t)$ through solution of the advection-diffusion equation
\begin{equation}\label{eq:advection-diffusion}
\begin{split}
u_t - k \Delta u + \bs{v} \cdot \nabla u &= 0 \text{ in } \mathcal{D} \times (0,T),\\
u(\cdot, 0) & = m \text{ in } \mathcal{D} ,\\
k \nabla u \cdot n & = 0 \text{ on } \partial\mathcal{D} \times (0,T),
\end{split}
\end{equation}
where $k=0.001$ is the diffusion coefficient and $T > 0$ is the final time. The domain $\mathcal{D} \subset \R^2$ is open and bounded with boundary $\partial \mathcal{D}$ depicted in \cref{fig:velocity}. The initial condition $m$ is an infinite-dimensional random parameter field in $\mathcal{D}$, which is to be inferred.  The velocity field $\bs{v} \in \R^2$ is obtained as the solution of the steady-state Navier--Stokes equations with Dirichlet boundary condition,
 \begin{equation}\label{eq:Navier-Stokes}
 \begin{split}
-\frac{1}{\text{Re}} \Delta \bs{v} + \nabla q + \bs{v} \cdot \nabla \bs{v} & = 0 \text{ in } \mathcal{D} ,\\
\nabla \cdot \bs{v} & = 0 \text{ in } \mathcal{D} ,\\
\bs{v} & = \bs{g} \text{ on } \partial\mathcal{D},
 \end{split}
\end{equation}
where $q$ represents the pressure field and the Reynolds number $\text{Re} = 50$. The Dirichlet boundary data $\bs{g} \in \R^2$ are prescribed as $\bs{g} = (0, 1)$ on the left wall of the domain, $\bs{g} = (0, -1)$ on the right wall, and $\bs{g} = (0, 0)$ elsewhere.
\begin{figure}[tbhp]
\footnotesize{
  \centering
  \subfloat[velocity field $\bs{v}$]{\label{fig:v}\includegraphics[width=0.25\linewidth]{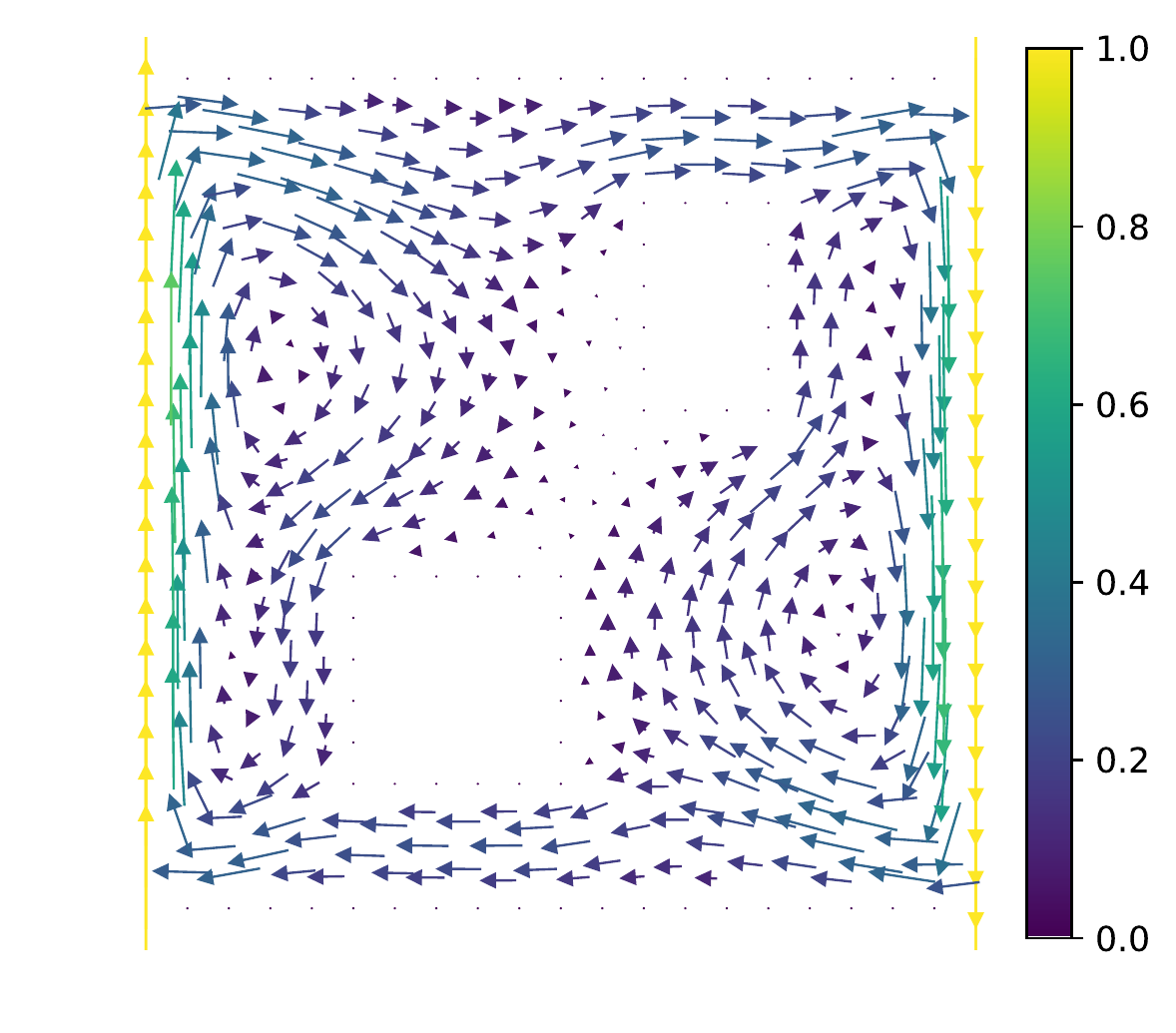}}
  \subfloat[initial condition $m_{\text{true}}$]{\label{fig:init_cond}\includegraphics[width=0.25\linewidth]{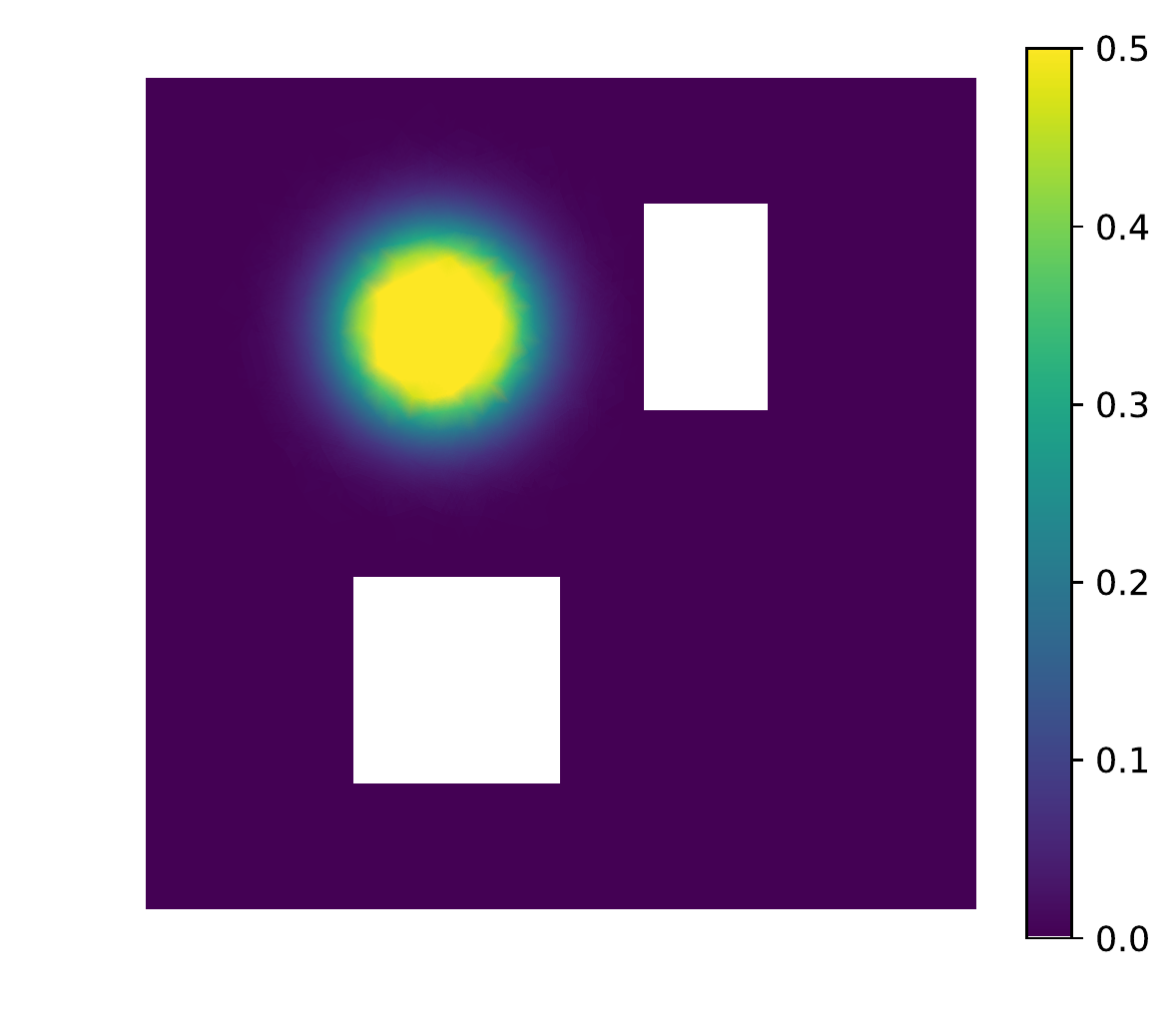}}
    \subfloat[9 candidates]{\label{fig:obs9}\includegraphics[width=0.225\linewidth]{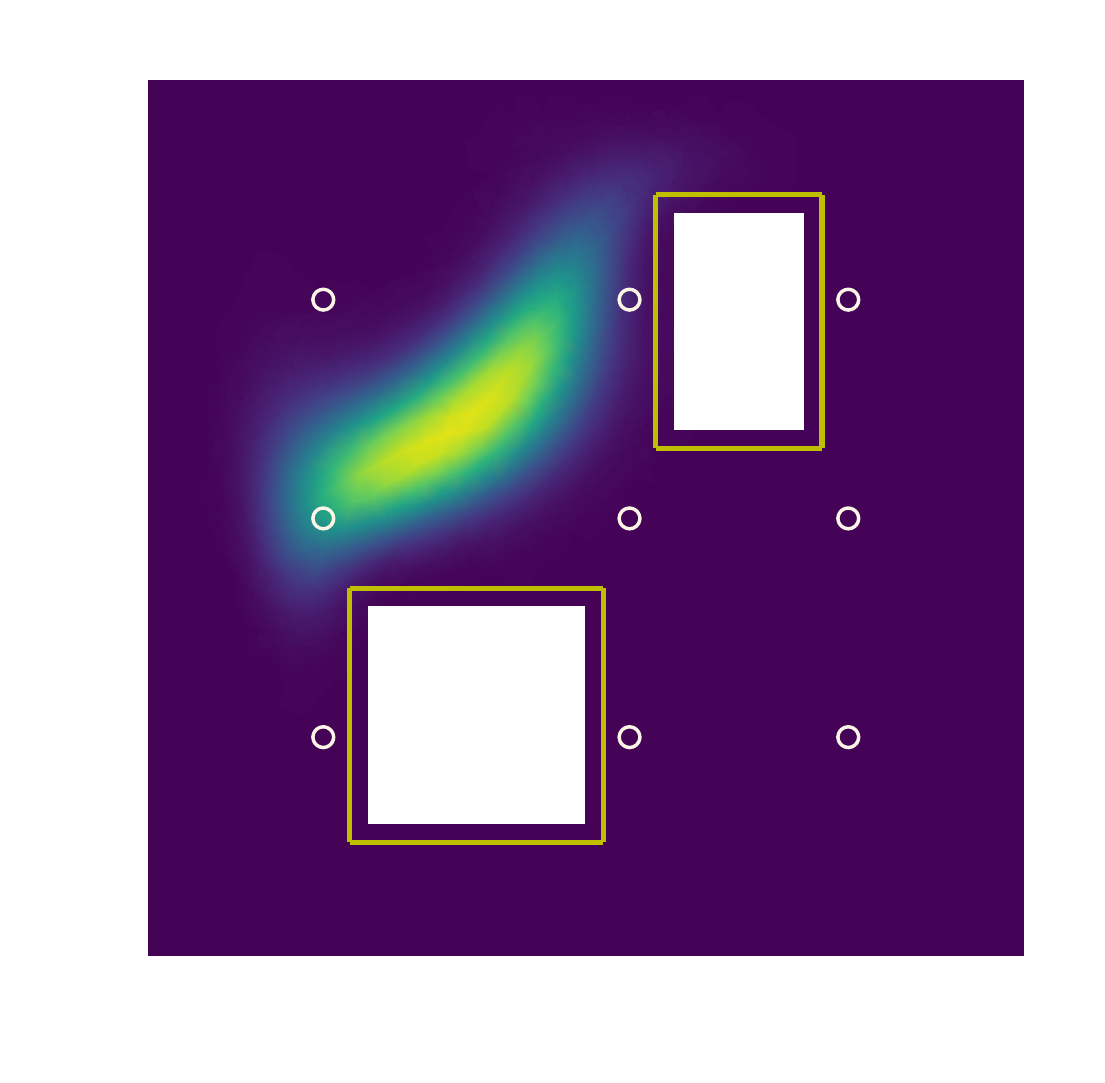}}
  \subfloat[75 candidates]{\label{fig:obs75}\includegraphics[width=0.225\linewidth]{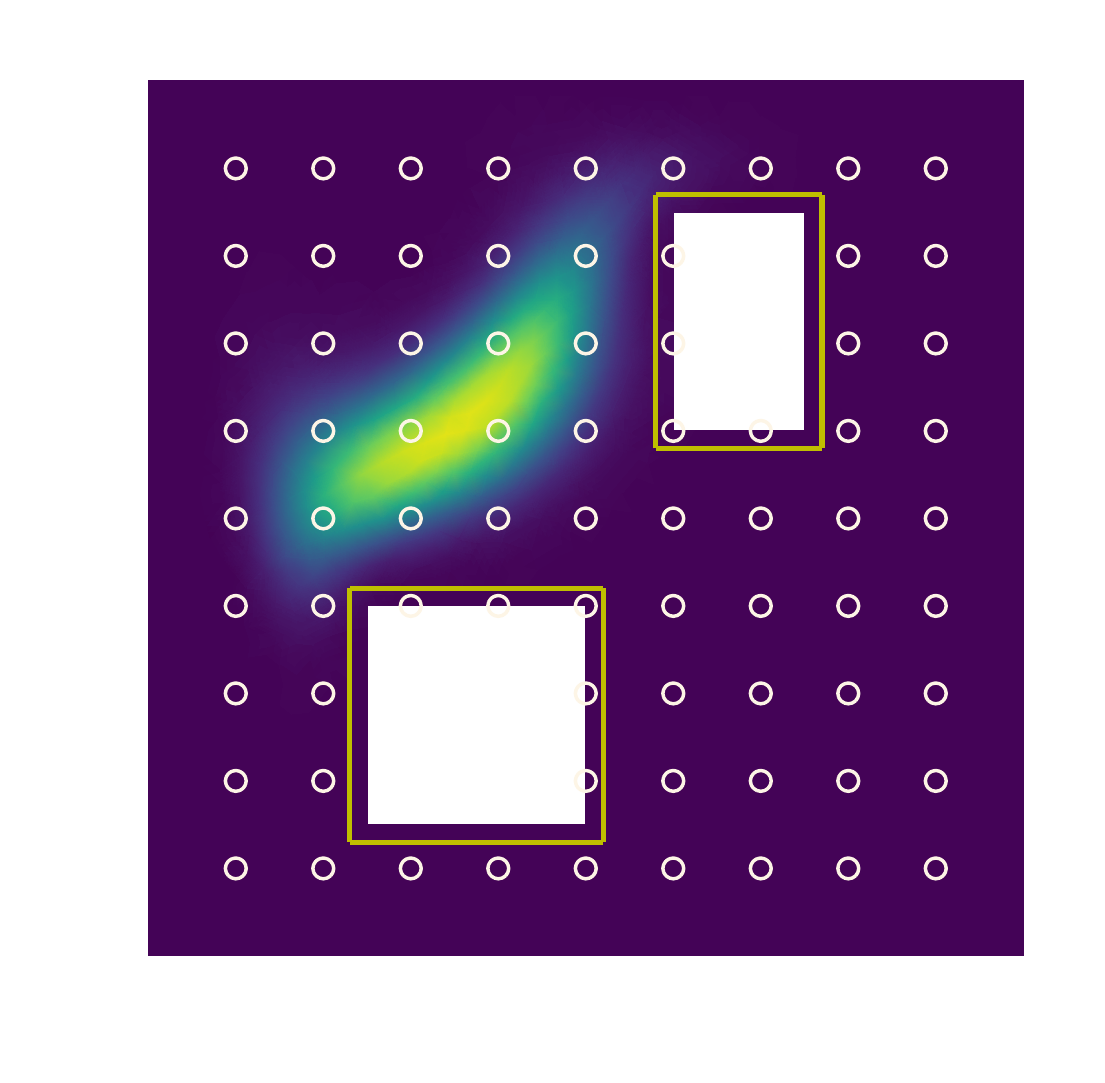}}
  \caption{The domain $\mathcal{D}$ is $[0,1]^2$ with two rectangular blocks ($[0.25,0.5]\times[0.15,0.4], [0.6,0.75]\times[0.6,0.85]$) removed. Data of contaminant concentration at time $T = 0.8$, obtained as the solution of \eqref{eq:advection-diffusion} at the initial condition as shown. The QoI maps ($\bP_1,\bP_2,\bP_3$) are the averaged solution within the lines along the left, right, and both buildings. Candidate sensor locations are shown in circles.}
  \label{fig:velocity}
  }
  \vspace{-0.2cm}
\end{figure}
We consider a Gaussian prior for the parameter $m \sim \mathcal{N}(\ipar_{\text{pr}}, \mathcal{C}_{\text{pr}})$ with mean $\ipar_{\text{pr}}$ and covariance operator $\mathcal{C}_{\text{pr}} = \mathcal{A}^{-2}$, where the elliptic operator $\mathcal{A} = -\gamma \Delta + \delta I$ (with Laplacian $\Delta$ and identity $I$) is equipped with Robin boundary condition $\gamma \nabla m \cdot \mathbf{n} + \beta m$ on $\partial \mathcal{D}$. Here $\gamma, \delta > 0$ control the correlation length and variance of $m$ \cite{DaonStadler18}. In our numerical test, we set $\prmean = 0.25$, $\gamma = 1, \delta = 8$. We synthesize a ``true" initial condition $m_{\text{true}} = \min(0.5, \exp(-100\norm{x-[0.35,0.7]}^2)$ as the contaminant source (\cref{fig:init_cond}). To solve the PDE model, we use an implicit Euler method for temporal discretization with $N_t$ time steps, and a finite element method for spatial discretization,  resulting in a  $d_m$-dimensional discrete parameter $\bipar \sim \mathcal{N}(\bipar_{\text{pr}}, \bCpr)$, with  $\bipar_{\text{pr}}, \bCpr$ denoting finite element discretizations of $\ipar_{\text{pr}}, \mathcal{C}_{\text{pr}}$, respectively.

The solution of the PDE for $d_m = 2023$ and $N_t = 40$ at the observation time $T = 0.8$ and $d$ candidate sensor locations are also shown in \cref{fig:obs9} and \cref{fig:obs75}, at which we observe the contaminant concentration $u$. The linear map $\bF$ is defined by the predicted data, i.e., the concentrations at the selected sensors. Finally, we take the QoI as an averaged contaminant concentration at time $t_{\text{pred}}$ within a distance $\delta = 0.02$ from the boundaries of either the left, the right, or both buildings, with corresponding QoI maps denoted as $\bP_1,\bP_2,\bP_3$ (see  \cref{fig:obs9} and \cref{fig:obs75}). 

%----------------------------------------------------------------------------------------
%	Numerical results
%----------------------------------------------------------------------------------------
\subsection{Numerical results}
We first consider the case of a small number of candidate sensors, for which we can use exhaustive search to find the optimal sensor combination and compare it with the sensors chosen by the standard and swapping greedy algorithms. Specifically, we use a grid of $d = 9$ candidate locations $\{s_i\}_{i=0}^9$ ($x_i \in \{0.2,0.55,0.8 \}\times \{0.25,0.5,0.75 \}$) as shown in \cref{fig:obs} (left) with the goal of choosing $r = 2, 3, 4, 5, 6, 7, 8$ sensors for the QoI prediction time $t_{\text{pred}} = 1.0$. We compute the matrices $\bH_d^{\rho}$ and $\Delta \bH_d$ (of size $9\times 9$) without low-rank approximation since they are small.
\begin{figure}[ht]
  \centering
 \subfloat[{\small $\bP_1$}]{\label{fig:11}\includegraphics[width=0.33\linewidth]{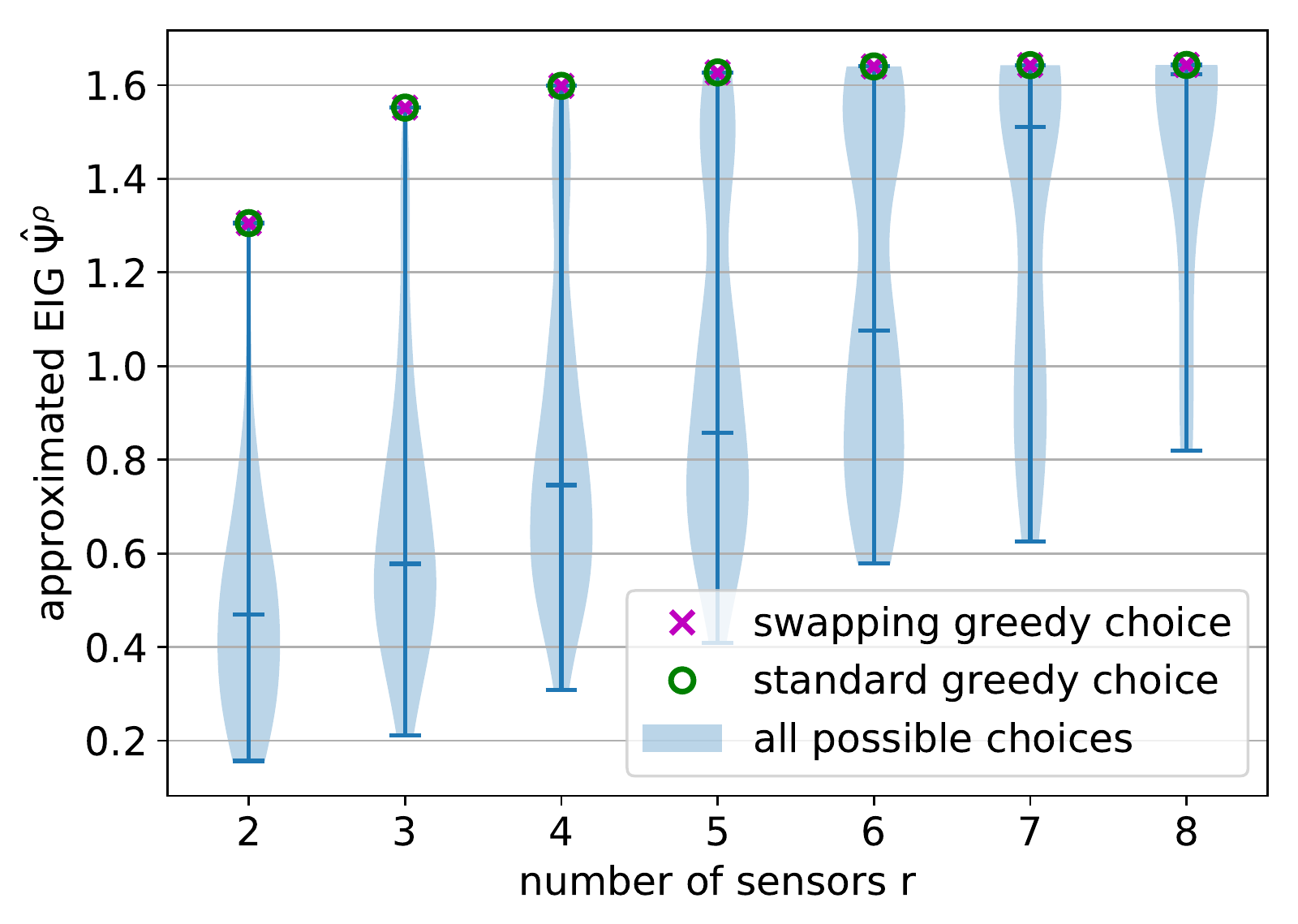}}
  \subfloat[{\small $\bP_2$}]{\label{fig:12}\includegraphics[width=0.33\linewidth]{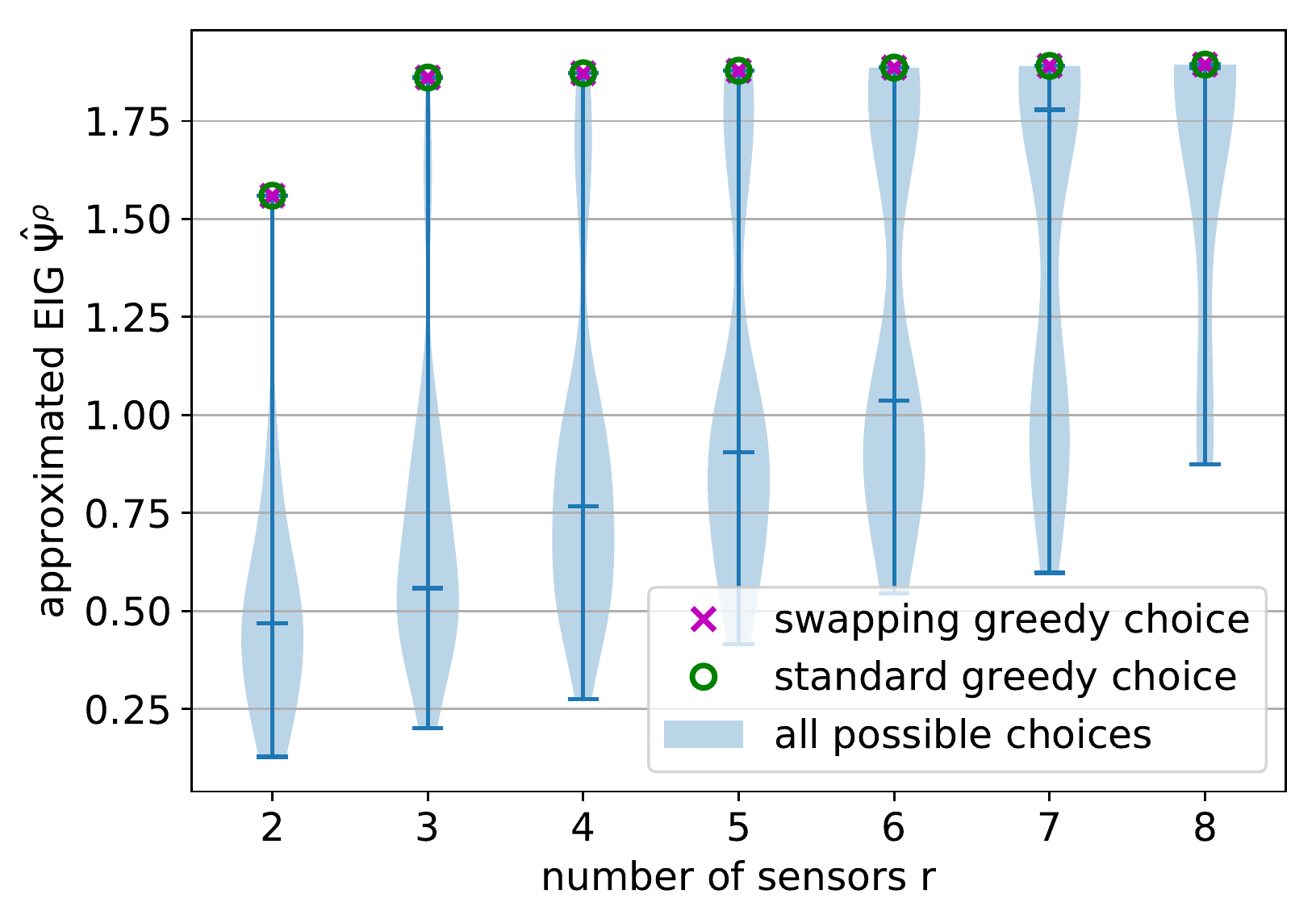}}
 \subfloat[{\small $\bP_3$}]{\label{fig:13}\includegraphics[width=0.33\linewidth]{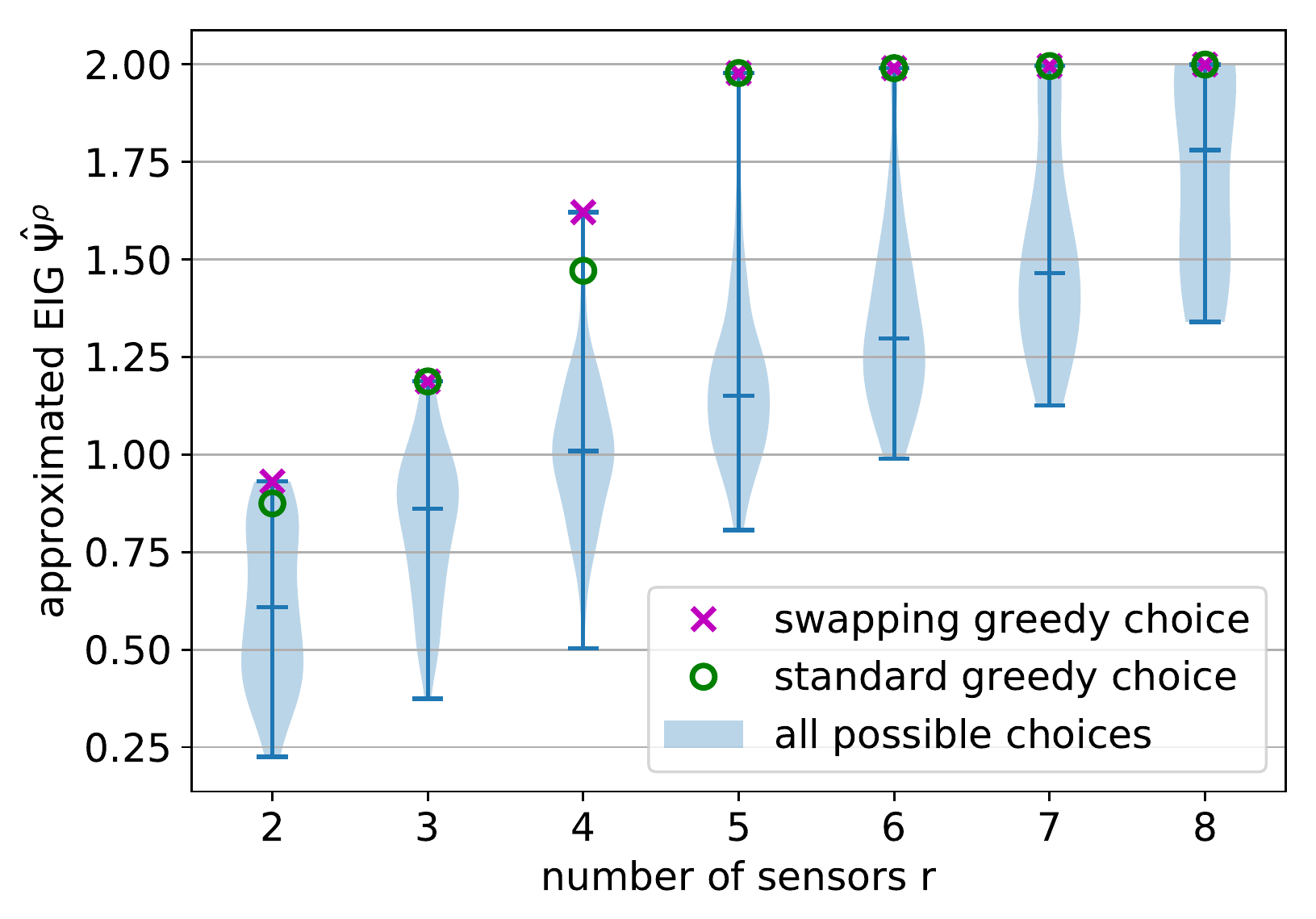}}
\caption{\small Approximate EIG $\hat{\Psi}^{\rho}$ at $r$ sensors chosen by the standard and swapping greedy algorithms, and the distribution of $\hat{\Psi}^{\rho}$ at all possible combinations of $9$ candidate sensors. The three plots are for the QoI maps $\bP_1,\bP_2$, and $\bP_3$. }
\label{fig:r9}

\end{figure}

We can see from \cref{fig:r9} that for QoI maps $\bP_1$ and $\bP_2$, both greedy algorithms find the optimal design, while for $\bP_3$ with $r=2,4$, only swapping greedy finds the optimal design. Moreover, an increase in $r$ leads to diminishing returns, as the gain in information about the QoI from additional sensors saturates. We see that $\sim$3 sensors is sufficient for either building, whereas 5 is sufficient for both. 
% limited extra information is obtained as EIG $\hat{\Psi}^{\rho}$ becomes stable, which indicates that a small $r$ is sufficient in this case. 

Next we consider the case of the 75 candidate sensors depicted in \cref{fig:obs} (right). Exhaustive search across all sensor combinations is not feasible in this case; instead, we compare the best EIG from $200$ random designs with those obtained by the greedy algorithms. We seek the $r$ optimal sensors, $r = 5,10,15,20,25,30,40,50,60$, from among the 75 candidates. Results are shown in \cref{fig:r75}.
\begin{figure}[ht]
\vskip 0.2in
  \centering
 \subfloat[{\small $\bP_1$}]{\label{fig:21}\includegraphics[width=0.33\linewidth]{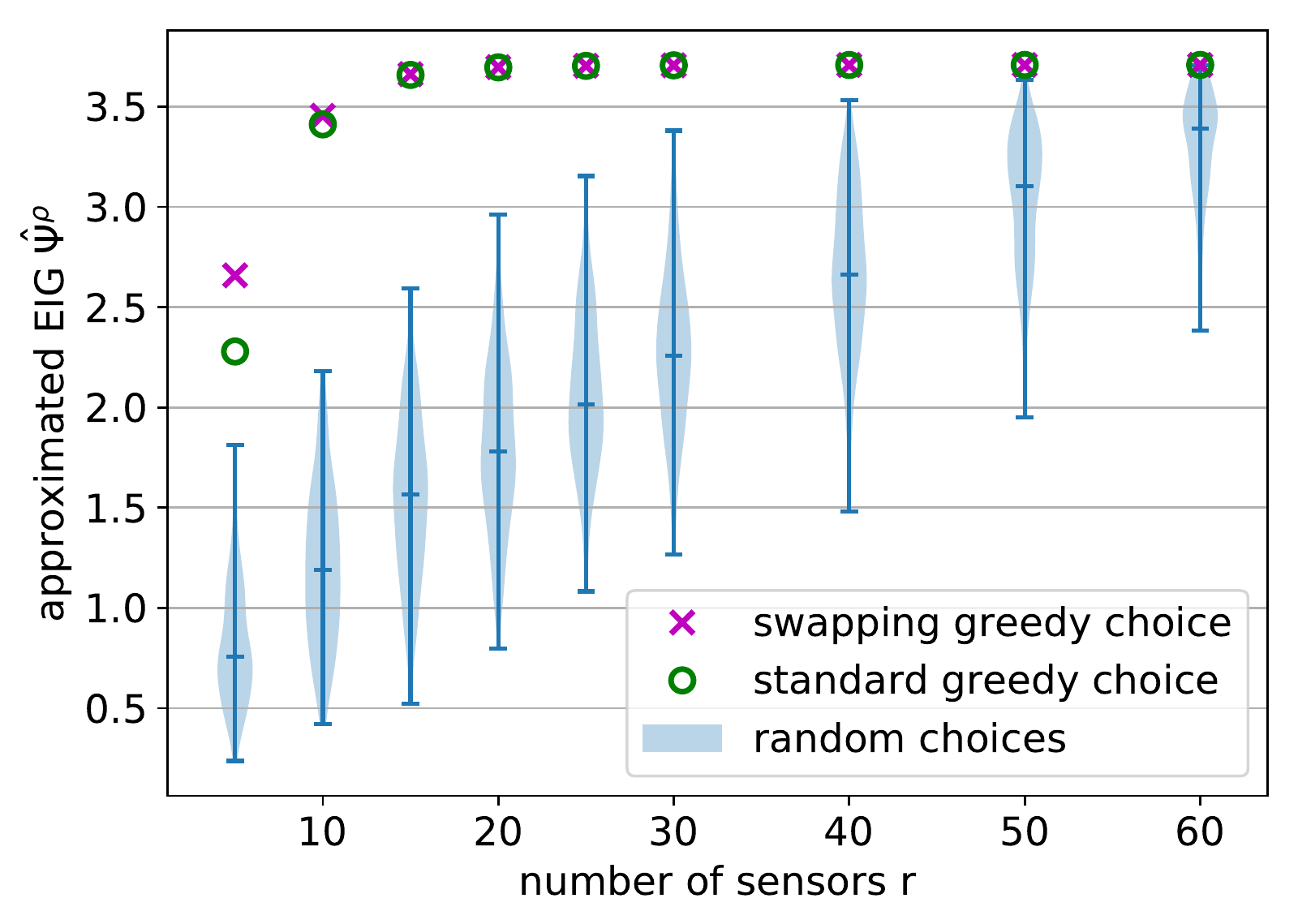}}
  \subfloat[{\small $\bP_2$}]{\label{fig:22}\includegraphics[width=0.33\linewidth]{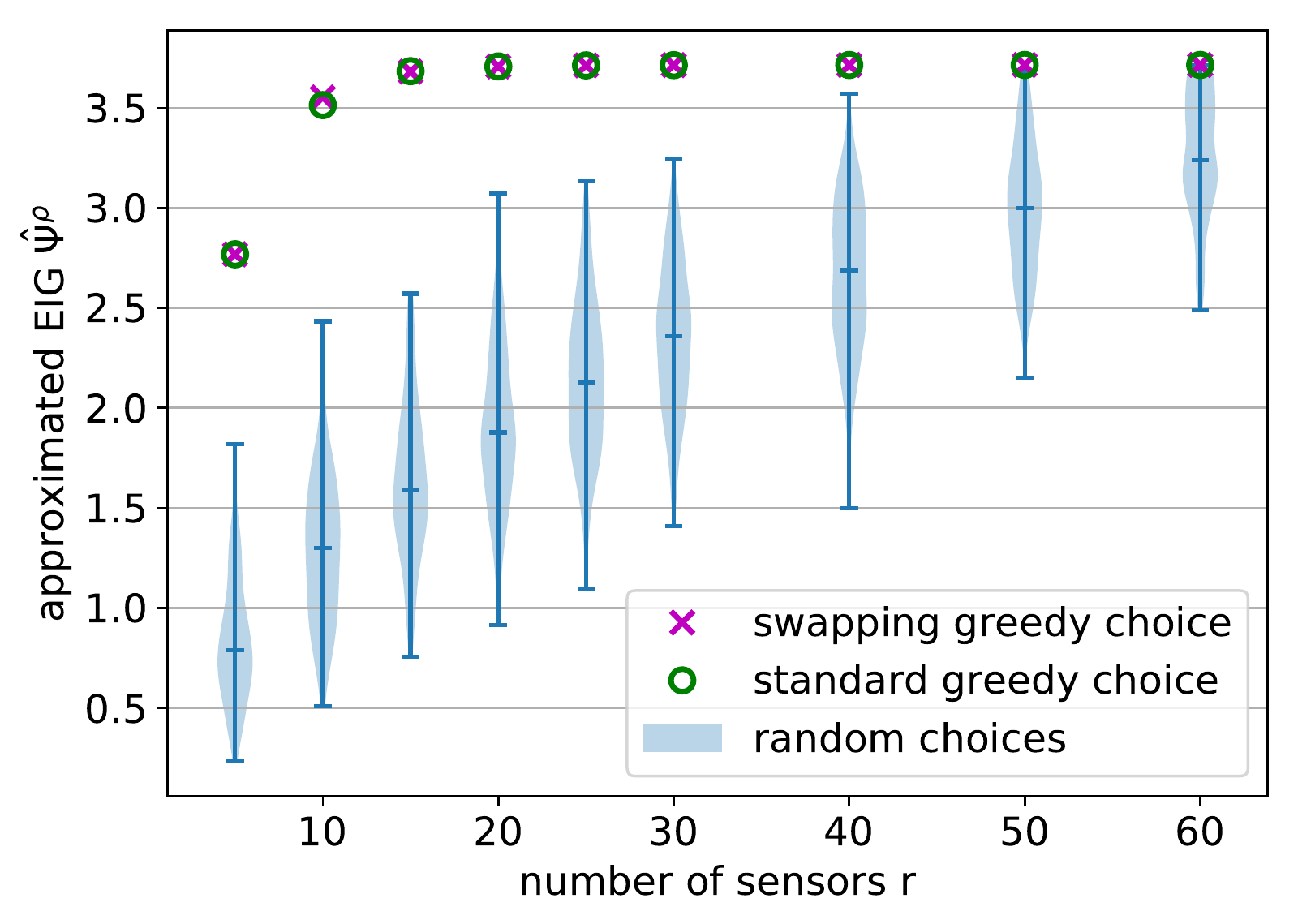}}
 \subfloat[{\small $\bP_3$}]{\label{fig:23}\includegraphics[width=0.33\linewidth]{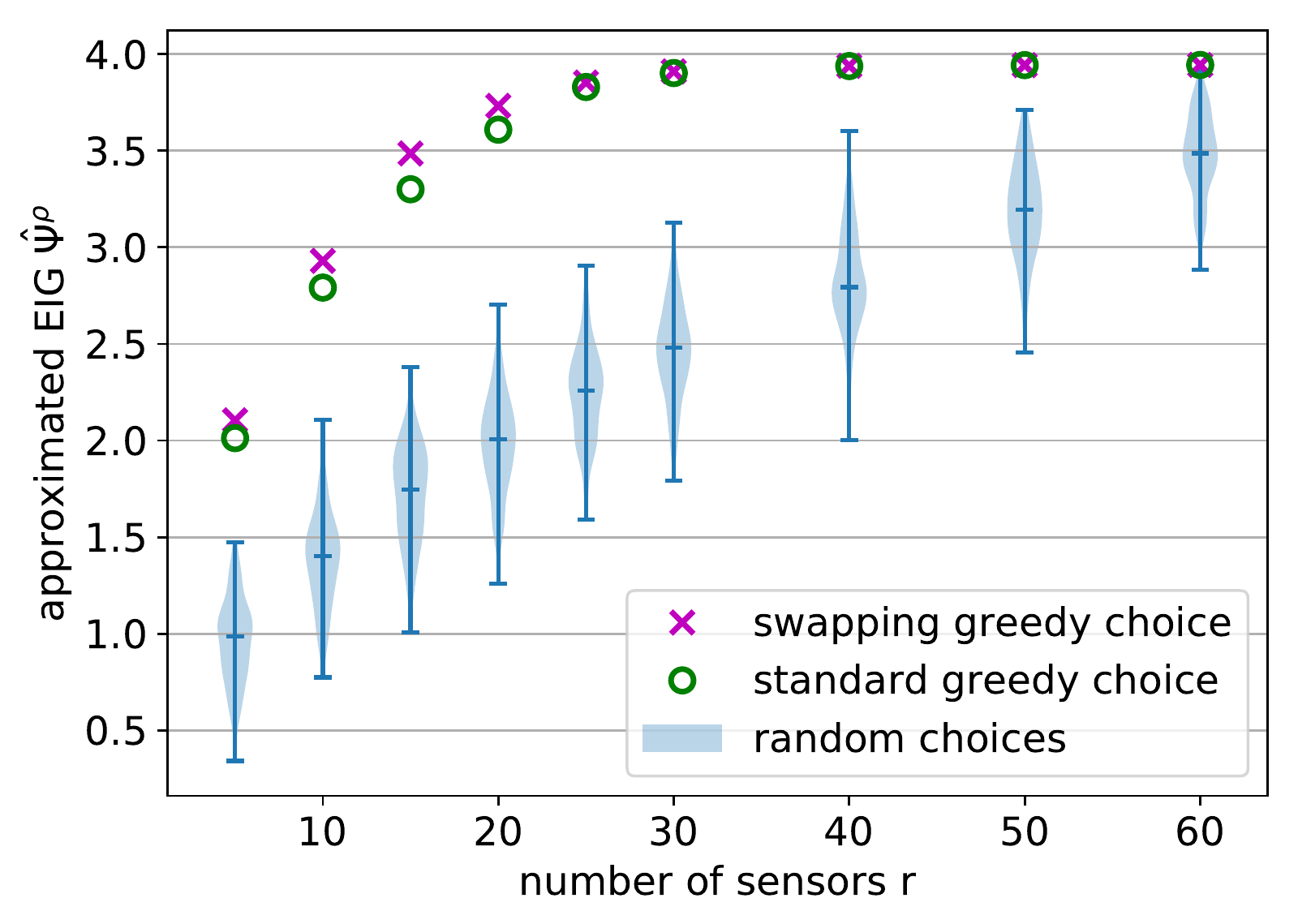}}
\caption{\small Approximate EIG $\hat{\Psi}^{\rho}$ for $r$ out of 75 sensors, found by the standard and swapping greedy algorithms, compared with the distribution of $\hat{\Psi}^{\rho}$  for 200 randomly-chosen sets from the $75$. The three plots are for the QoI maps $\bP_1,\bP_2$ and $\bP_3$ . 
}
\label{fig:r75}

\end{figure}
We see that both greedy algorithms find designs with larger EIG than all random choices. Moreover, for small $r$, the swapping greedy algorithm finds better designs than the standard greedy. For large $r$, both greedy algorithms can find designs with similar EIG. In fact, multiple designs with similar EIG become more likely with larger $r$. 

To demonstrate the reduction of computational cost achieved by the offline-online decomposition, we report the total number of EIG evaluations, the number of swapping loops, and the number of swaps of the swapping greedy algorithm (\cref{alg:swapping}) in \cref{table:1} for $75$ candidate sensors with different target number of sensors. We see that the number of loops at convergence is mostly $3$. We observe in the experiments that most of the swaps take place in the first loop, followed by a smaller number of swaps in the second loop resulting in slight sensor adjustments. There are no swaps in the last loop, which we require as a convergence  criterion. As a result of the offline-online decomposition \cref{eq:OffOn}, which relieves the (thousands of) EIG evaluations of expensive PDE solves once the low-rank approximation \eqref{eq:svd} is built, we achieve over 1000X speedup. This is because the PDE solves overwhelmingly dominate the overall cost, and because the offline decomposition is computed at a cost comparable to one direct EIG evaluation by \eqref{eq:EIGz}. 

\begin{table}[t]
\caption{\small Number of swapping loops (\#LOOPS), swaps (\# SWAPS), and EIG evaluations (\# EIG EVAL) for different numbers of $r$ selected sensors out of 75 candidates. Results are reported for \cref{alg:swapping} for the goal $\bP_1$.}
\label{table:1}
\vskip -0.5in
\begin{center}
\begin{small}
\begin{sc}
\begin{tabular}{|c|c|c|c|c|c|}
\toprule
$r$ & $5$ & $10$ & $15$ & $20$ & $25$\\
\hline
\#loops & $3$ & $3$& $3$ & $3$ & $3$\\
\hline
\#swaps & $41$ & $73$ & $124$ & $164$ & $190$\\
\hline
\#EIG eval & $1050$ & $1950$ & $2700$ & $3300$ & $3750$ \\
\hline
$r$ & $30$ & $40$ & $50$ & $60$ &\\
\hline
\#loops & $2$ & $3$ & $3$ & $3$ &\\
\hline
\#swaps &  $194$ & $235$ & $199$& $119$ &\\
\hline
\#EIG eval &  $2700$ & $4200$ & $3750$& $2700$ &\\
\bottomrule
\end{tabular}
\end{sc}
\end{small}
\end{center}
\vskip -0.1in
\end{table}
\begin{figure}[ht]
\vskip 0.2in
  \centering
 \subfloat[{\small $\bP_1$ at $t_{\text{pred}}=1.$}]{\label{fig:31}\includegraphics[width=0.25\linewidth]{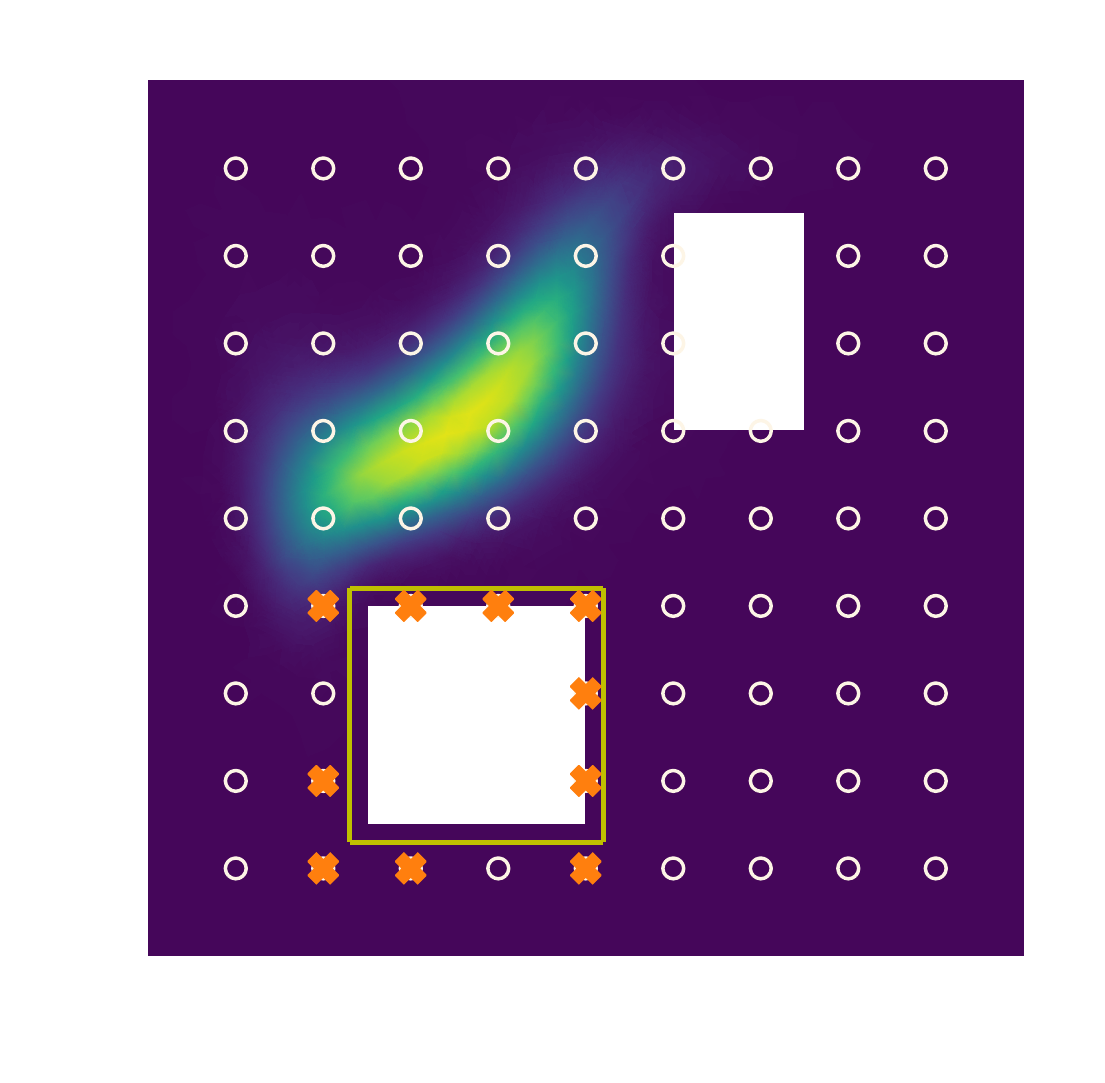}}
  \subfloat[{\small $\bP_2$ at $t_{\text{pred}}=1.$}]{\label{fig:32}\includegraphics[width=0.25\linewidth]{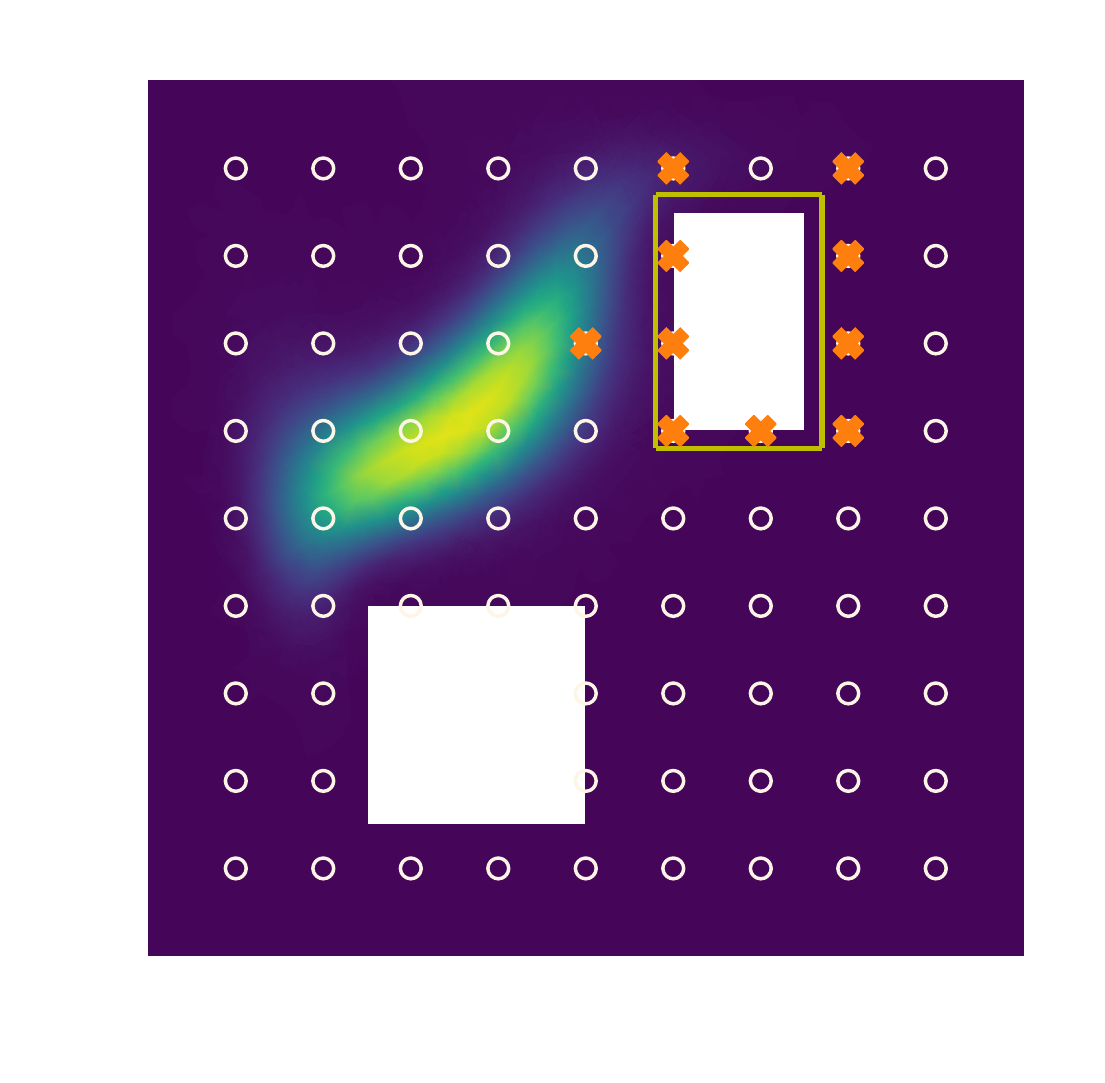}}
 \subfloat[{\small $\bP_3$ at $t_{\text{pred}}=1.$}]{\label{fig:33}\includegraphics[width=0.25\linewidth]{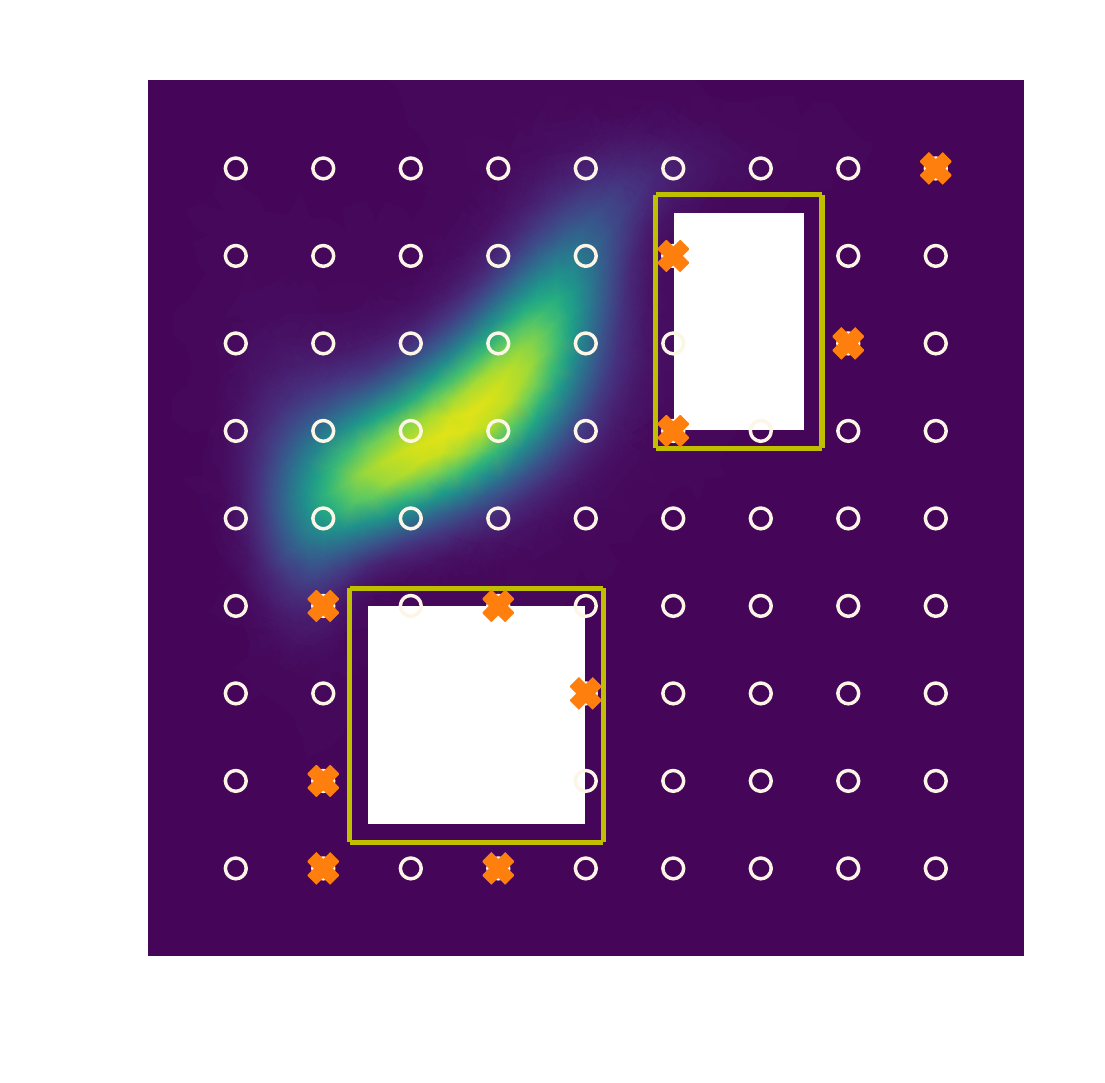}}\\
 \vspace{-0.3cm}
  \subfloat[{\small $\bP_1$ at $t_{\text{pred}}=2.$}]{\label{fig:34}\includegraphics[width=0.25\linewidth]{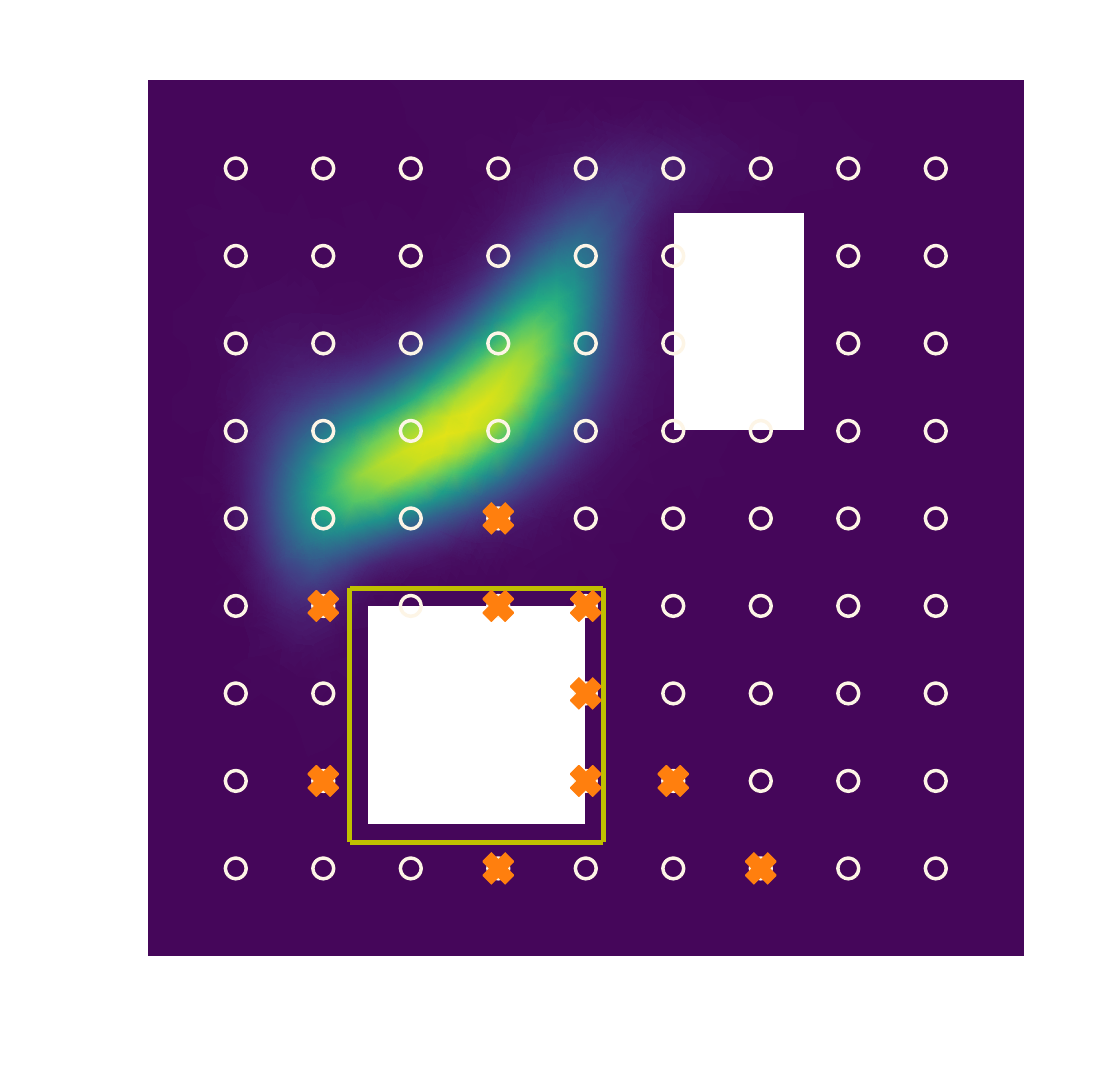}}
  \subfloat[{\small $\bP_1$ at $t_{\text{pred}}=4.$}]{\label{fig:35}\includegraphics[width=0.25\linewidth]{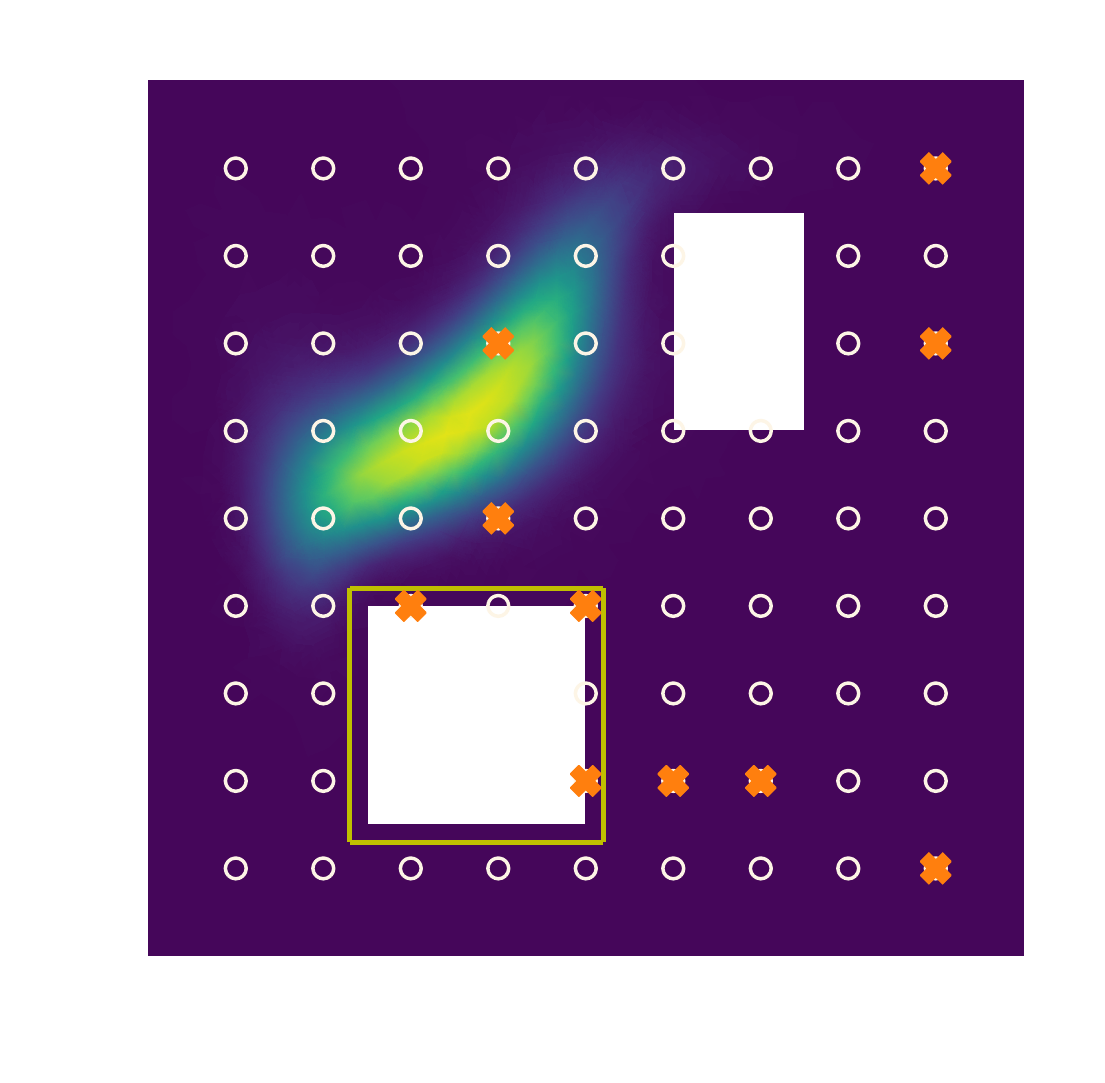}}
 \subfloat[{\small $\bP_1$ at $t_{\text{pred}}=8.$}]{\label{fig:36}\includegraphics[width=0.25\linewidth]{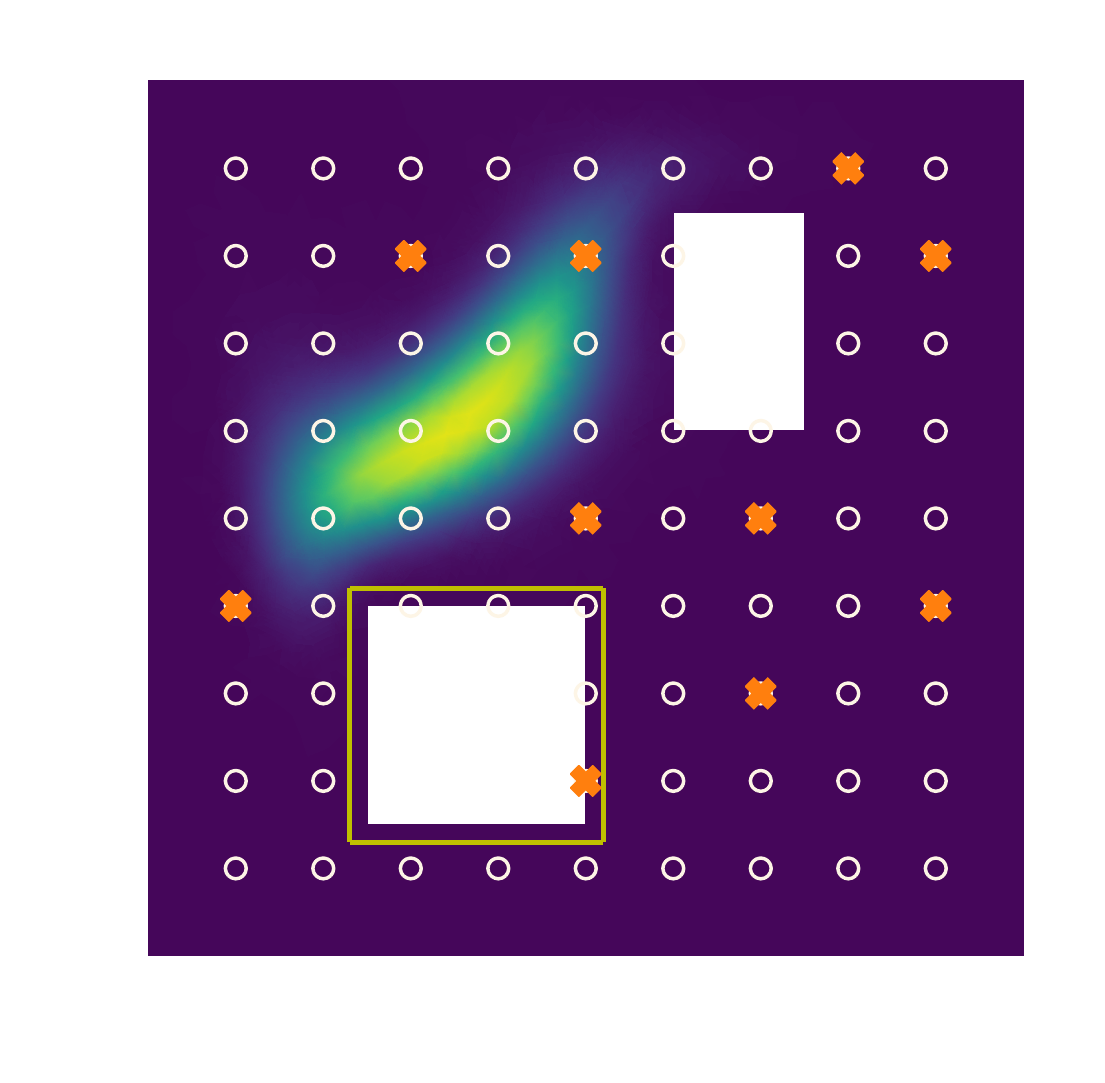}} 
\caption{\small Sensor locations chosen by the swapping greedy algorithm for $10$ out of $75$ candidates for the parameter-to-QoI maps $\bP_1,\bP_2,\bP_3$ at time $t_{\text{pred}}=1$ and also $\bP_1$ at time  $t_{\text{pred}}=2, 4, 8$.}
\label{fig:75t1}

\vskip -0.2in
\end{figure}

\Cref{fig:75t1} illustrates the effect of the goal of maximizing information gain for the QoIs from optimally placed sensors. Specifically, for the parameter-to-QoI maps $\bP_1, \bP_2, \bP_3$ that quantify the average contaminant concentration at time $t_{\text{pred}} = 1$ around  left,  right, and both blocks, the goal-oriented OED finds the sensors depicted 
% around left, right, and both blocks too, respectively, as can be observed from 
in the first row. For $\bP_1$ at longer prediction times $t_{\text{pred}} = 1, 2, 4, 8$, we see in the bottom row of \cref{fig:75t1} that the optimal sensors are no longer placed in the immediate vicinity of the building, but instead are increasingly dispersed to better detect the now more diffused field. Finally, the ability of GOOED to reduce the posterior variance in the initial condition field is depicted in \cref{fig:post} for different goals $\bP_1,\bP_2,\bP_3$. Compared to a random design (lower right), the three optimal designs lead to lower variance surrounding regions of interest. 
\begin{figure}[ht]
        \centering
  \subfloat[{\small  optimal design for $\bP_1$}]{\label{fig:41}\includegraphics[width=0.24\linewidth]{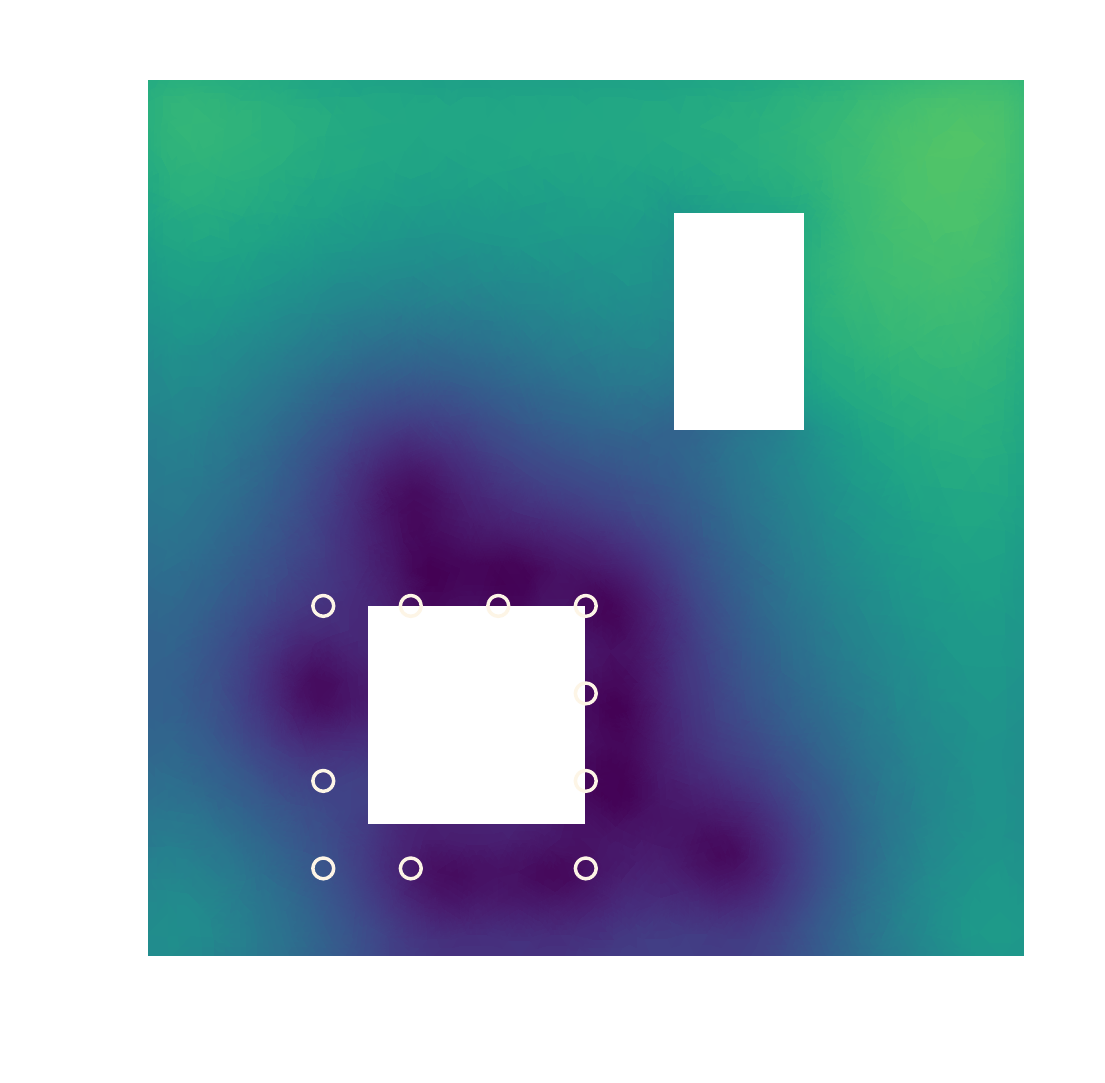}}
  \subfloat[{\small  optimal design for $\bP_2$}]{\label{fig:42}\includegraphics[width=0.24\linewidth]{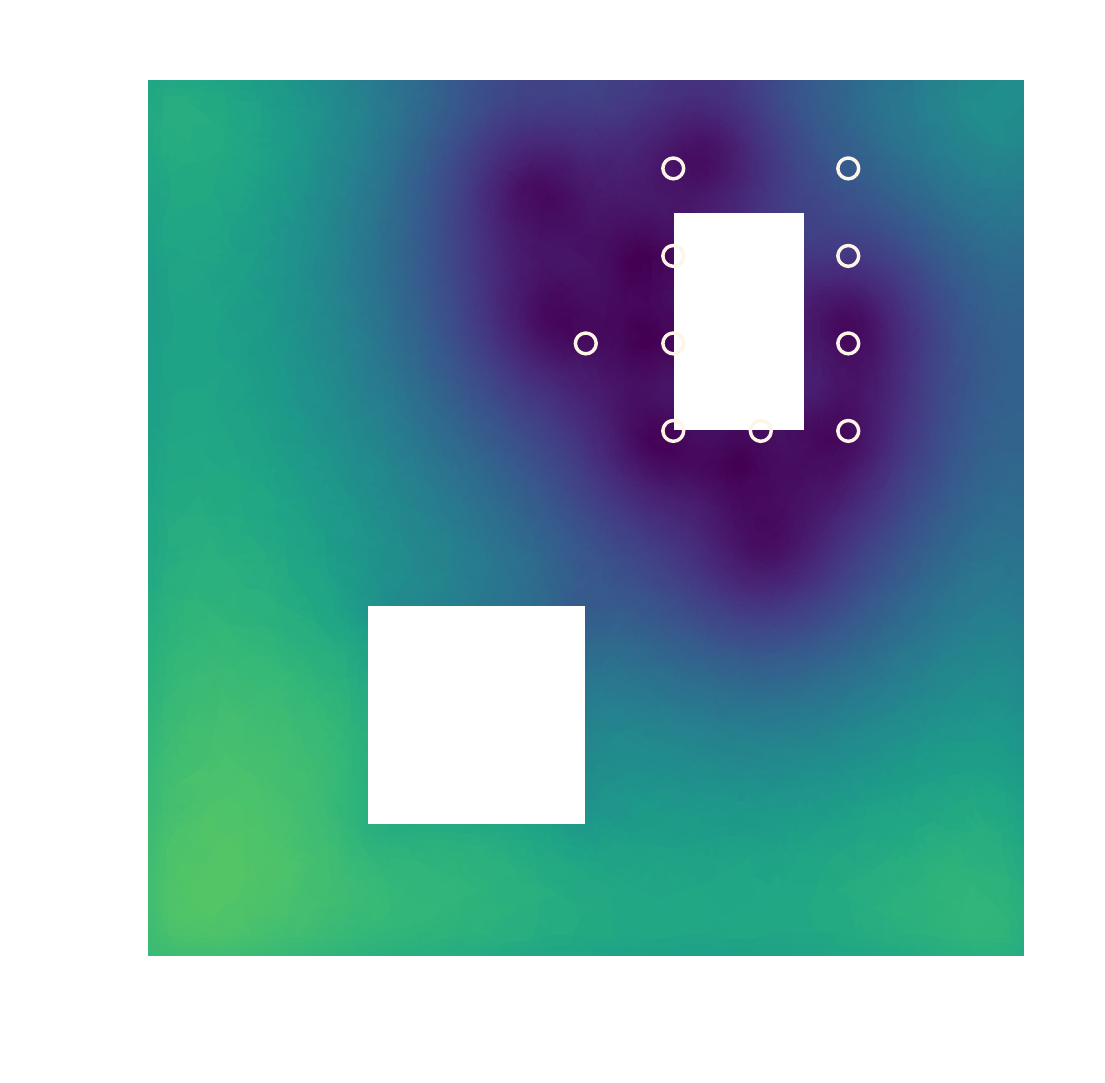}}
 \subfloat[{\small optimal design for $\bP_3$}]{\label{fig:43}\includegraphics[width=0.24\linewidth]{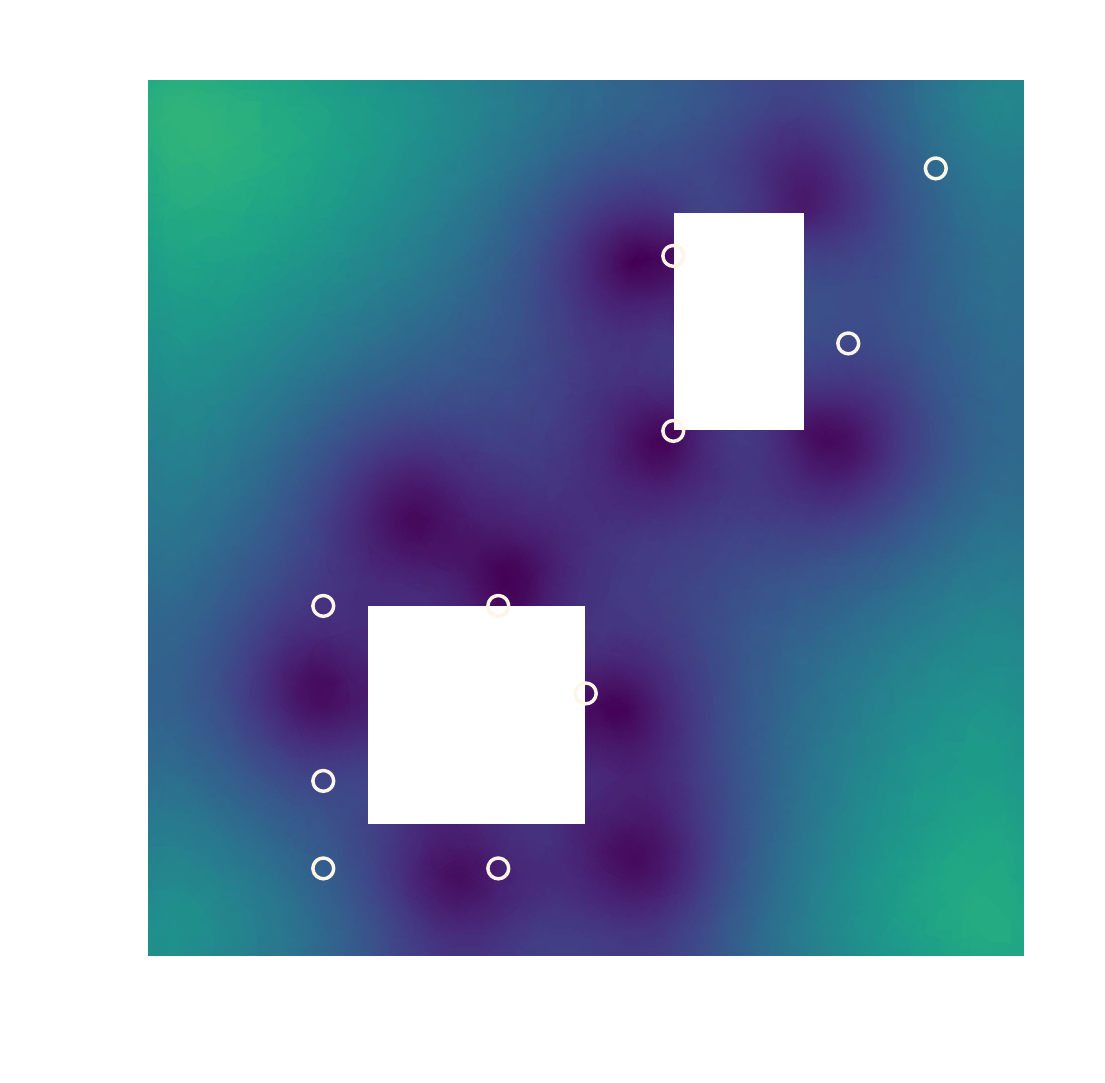}} 
  \subfloat[{\small random design}]{\label{fig:44}\includegraphics[width=0.24\linewidth]{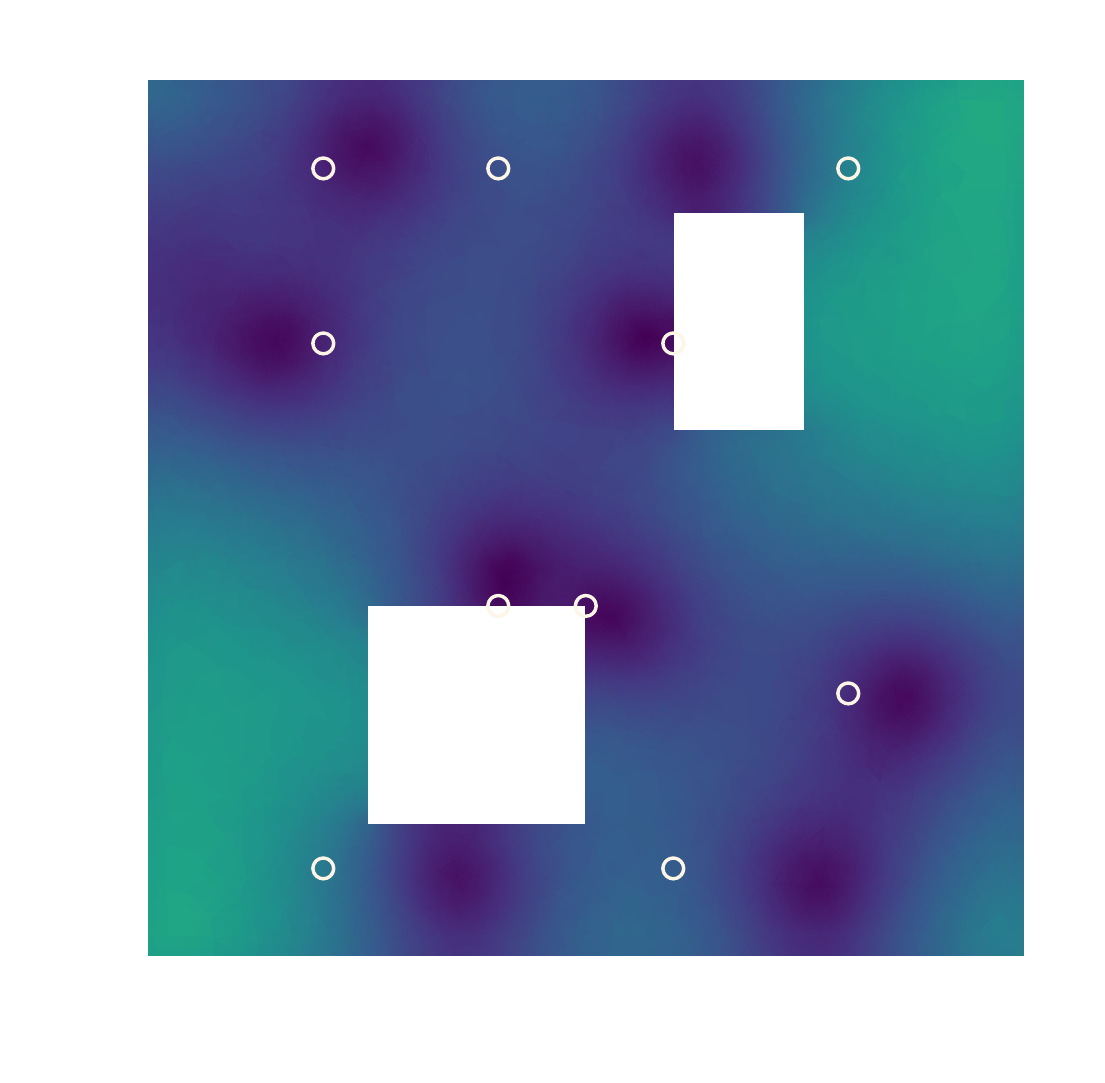}} 
        \caption{\small Pointwise posterior variance of the parameter at optimal designs for goals $\bP_1$,$\bP_2$, $\bP_3$, compared to a random design, for 10 sensors. The darker regions represent lower variance.}
        \label{fig:post}
    \vskip -0.2in
    \end{figure}
\subsection{Scalability w.r.t.\ parameter and data dimensions}
\label{sec:scalability}
Here we demonstrate the fast decay of the eigenvalues of $\bH_d^{\rho}$ and $\Delta \bH_d$ with respect to the parameter and data dimensions, as exploited by the algorithms of \cref{sec:svd}.
%With the computational framework described in Section \cref{sec:GOOED}, we have analyzed the computational complexity in terms of the number of large-scale model evaluations depends on the rank $k$ and $l$ in the low-rank approximation of $\bH_d^{\rho}$ and $\Delta \bH_d$ respectively. 
For $\bH_d^{\rho}$ defined in \cref{eq:Hdrho}, we have  rank($\bH_d^{\rho}$) $\leq \min(p,d)$ with QoI dimension $p$ and data dimension $d$. In practice, the QoI is often an averaged quantity with small  $p$, so the rank of $\bH_d^{\rho}$ is also small. In our tests we have rank($\bH_d^{\rho}$) $ = p = 1$. For $\Delta \bH_d = \bH_d-\bH_d^{\rho}$ with $\bH_d = \bF_d \bCpr \bF^*$, the spectrum of $\Delta \bH_d$ depends on that of $\bH_d$, which typically exhibits fast decay due to ill-posedness of inverse problems. 
% as it is the data-misfit Hessian for all candidate sensors. 
%In our experiment, we consider scalar predicted QoIs, thus the rank $k$ of $\bH_d^{\rho}$ is always $1$, 
% We investigate the the dependence of eigenvalue decay of $\Delta \bH_d$ on parameter dimension $d_m$ and data dimension (the number of candidate sensor locations) $d$. 
As can be observed in the left plot of \cref{fig:scale}, the eigenvalues of $\Delta \bH_d$ decay very rapidly and independently of the parameter dimension, which implies that the required number of PDE solves is small and independent of the parameter dimension while achieving the same absolute accuracy of the approximate EIG by \cref{thm:low-rank-EIG}. The right plot in \cref{fig:scale} also illustrates rapid decay of eigenvalues, as well as  diminishing returns, with the increasing number of candidate sensors, suggesting that the number of PDE solves  is asymptotically independent of the data dimension for the same relative accuracy of the approximate EIG. These plots suggest that $O(100)$ PDE solves are required to accurately capture the information gained about the parameter field and QoI from the data, regardless of the parameter or sensor dimensions, when using randomized SVD ( \cref{alg:RSVD}). 
% Appendix \cref{sec:AppendixC}
% by Theorem \cref{thm:low-rank-EIG}.
\begin{figure}[ht]
\vskip 0.2in
\begin{center}
\centerline{\includegraphics[width=.45\columnwidth]{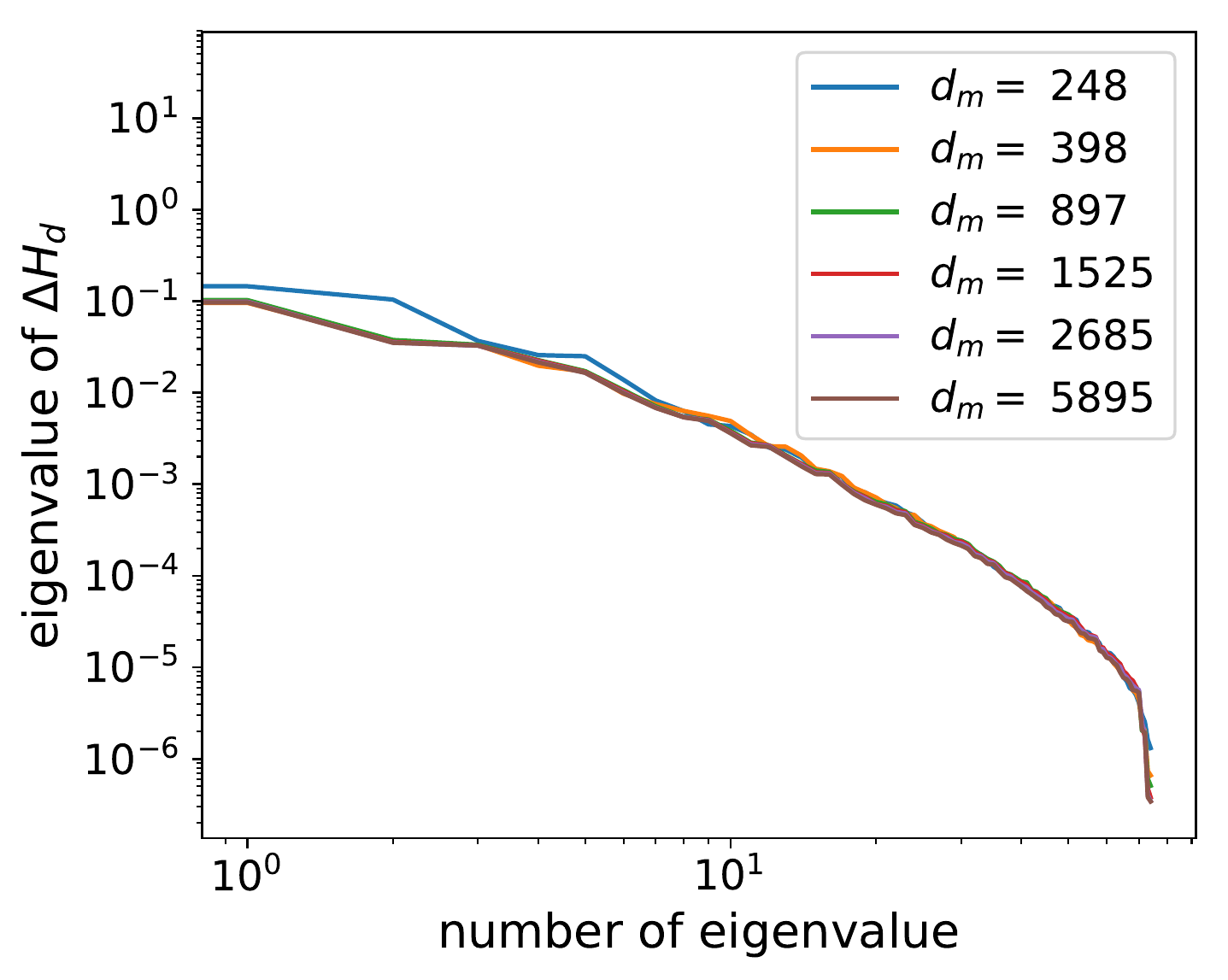}
\includegraphics[width=.45\columnwidth]{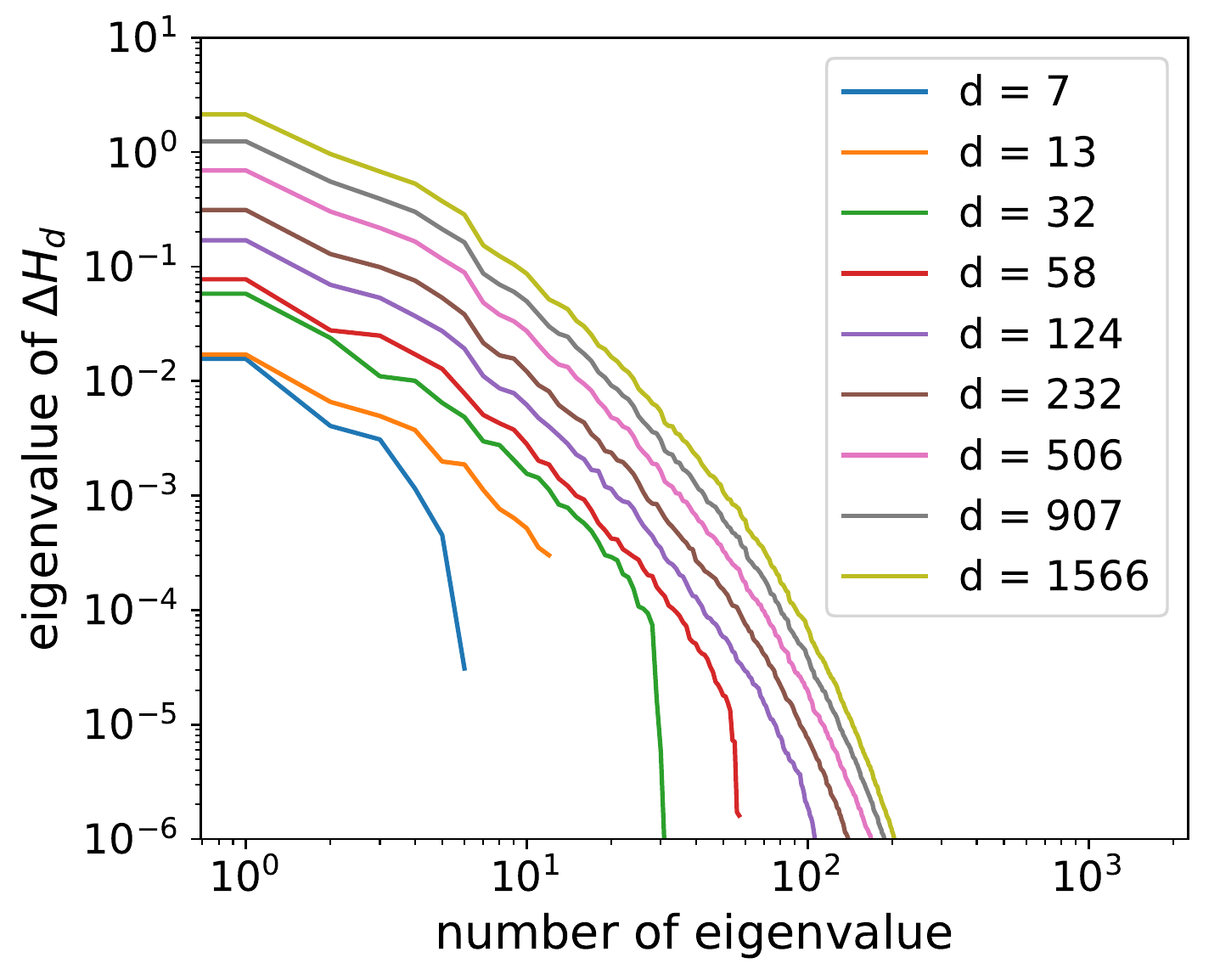}}
\caption{\small Decay of the eigenvalues of $\Delta \mathbf{H}_d$ with the increasing parameter dimension (left) and  data (candidate sensor locations) dimension (right).}
\label{fig:scale}
\end{center}
\vskip -0.2in
\end{figure}

\section{Conclusions}
We have developed a fast and scalable computational framework for goal-oriented linear Bayesian optimal experimental design governed by expensive models. Repeated fast evaluation of an (arbitrarily accurate) approximate EIG while avoiding model evaluations is made possible by an offline-online decomposition and low-rank approximation of certain operators informed by the parameter, data, and predictive goals of interest. Scalability, as measured by parameter- and data-dimension independence of the number of model evaluations, is achieved by carefully exploiting the GOOED problem's intrinsic low dimensionality as manifested by the rapid spectral decay of several critical operators. 
% we propose an efficient offline-online decomposition of the problem, where in the offline stage, we compute necessary quantities involving a limited number of large-scale parameter-to-observable map applications, in the online stage, we solve an optimization problem to maximize the expected information gain involving EIG evaluations without any forward model solves.  
To justify the low-rank approximation of these operators in computing the EIG, we proved an upper bound for the approximation error in terms of the operators' truncated eigenvalues. Moreover, we proposed a new swapping greedy algorithm that is demonstrated to be more effective than the standard greedy algorithm in our experiments. Numerical experiments with optimal sensor placement for Bayesian inference of the initial condition of an advection--diffusion PDE demonstrated over 1000X speedups (measured in PDE solves).  
Future work includes extension to nonlinear Bayesian GOOED problems with nonlinear parameter-to-observable maps and nonlinear parameter-to-QoI maps.

\appendix
	\section*{Appendix A:  Low-rank approximation}
	\label{sec:AppendixA}
To compute the low-rank approximations of $\Delta \bH_d$ and $\bH_d^{\rho}$ as described in \cref{sec:svd}, we present the randomized SVD algorithm for these two quantities. Recall the explicit forms of $\Delta \bH_d$ and $\bH_d^{\rho}$ as 
\begin{equation}
\mathbf{H}^{\rho}_d = \bF_d\bCpr\bP^*\mathbf{\Sigma}_{\text{pr}}^{-1}\bP\bCpr\bF_d^*,
\Delta \bH_d =  \mathbf{F}_d\bCpr\mathbf{F}_d^*-\bF_d\bCpr\bP^*\mathbf{\Sigma}_{\text{pr}}^{-1}\bP\bCpr\bF_d^*.    
\end{equation}

\begin{algorithm}[H]
	\small
	\caption{Randomized SVD to compute $\bH$ with low rank $k$}
	\label{alg:RSVD}
	\begin{algorithmic}[1]
		\STATE Generate  i.i.d. Gaussian matrix $\bs{\Omega} \in \mathbb
		R^{d \times (k+p)}$ with an oversampling parameter $p$ very small (e.g., $p = 10$). 
		% with  entries
		% with $k =$ numerical rank of $\bpHmisfit$ ($k \ll n$)
		\STATE  Compute $\bs{Y} = \mathbf{H} \bs{\Omega} .$ % =  \bF_d  \, \bCpr\,  \bF_d^T \bs{\Omega}. 
		% =   (\mathcal{B}_d(u(\mathbf{m})))\, \bCpr\, ( \mathcal{B}_d(u(\mathbf{m})))^T \bs{\Omega} 		%\\ \small{ each product of $\bpHmisfit$ with a column of $\bs{\Omega}$ entails a	tangent linear model solve and a 2nd order adjoint solve }
		\STATE  Compute the QR factorization $\bs{Y} = \bs{Q}\bs{R}$ satisfying $\bs{Q}^T \bs{Q} = \bs{I}$.
		\STATE Compute $\bs{B} = \bs{Q}^T \mathbf{H}  \bs{Q}$. %=   \bs{Q}^T \bF_d  \, \bCpr\,  \bF_d^T\bs{Q}
		% =  \bs{Q}^T   (\mathcal{B}_d(u(\mathbf{m})))  \, \bCpr\,   (\mathcal{B}_d(u(\mathbf{m})))^T\bs{Q} \in \mathbb{R}^{(k+p)\times (k+p)} 		
		\STATE Solve an eigenvalue problem for $\bs{B}$ such that $\bs{B} = \bs{Z} \bs{\Sigma} \bs{Z}^T$.
		\STATE Form $U_k = \bs{Q}\bs{Z}[1:k]$ and $\Sigma_k =\bs{\Sigma}[1:k, 1:k]$.
%		Form the low-rank approximation $\hat{ \Psi } \approx \bs{V} \bs{\Lambda} \bs{V}^T$, where $\bs{V} \in \mathbb R^{n \times k} := \bs{Q} \bs{Z}$ 
	\end{algorithmic}
\end{algorithm}
We see that this is a matrix-free eigensolver. Steps 2 and 4 represent $\Delta \bH_d$ action on $O(2(l+p))$ vectors and $\bH_d^{\rho}$ action on $O(2(k+p))$ vectors.  In terms of the total actions, it requires $2(2l+k+p)$ forward operator $\bF$ and $2(l+k+p)$ of its adjoint $\bF^*$, $2(k+l+p)$ prediction operator $\bP$ and its adjoint $\bP^*$. 
	
For the contaminant problem	given in \cref{sec:model-setting}, the concentration field $u(x,t)$ is given by 
\begin{equation}
\begin{split}
u_t - k \Delta u + \bs{v} \cdot \nabla u &= 0 \text{ in } \mathcal{D} \times (0,T),\\
u(\cdot, 0) & = m \text{ in } \mathcal{D} ,\\
k \nabla u \cdot \bs{n} & = 0 \text{ on } \partial\mathcal{D} \times (0,T),
\end{split}
\end{equation}
we can form the parameter-to-observable map $\bF \bipar$ as the discretized value of $\mathcal{B} u(m)$ where $\mathcal{B}$ is the pointwise observation operator.
The adjoint problem is a terminal value problem which can be solved backwards in time by the equation:
\begin{equation}
\begin{split}
-p_t - \nabla \cdot (p \bs{v})-k\Delta p &= \mathcal{B}^*\obs \text{ in } \mathcal{D} \times (0,T),\\
p(\cdot, T) & = 0 \text{ in } \mathcal{D} ,\\
(p\bs{v}+k\nabla p) \cdot \bs{n} & = 0 \text{ on } \partial\mathcal{D} \times (0,T).
\end{split}
\end{equation}
Then we can define the adjoint of the parameter-to-observable map $\bF^* \obs$ as the discretized value of $p(x,0)$ for any $\obs$.	
\begin{comment}	
	\section*{Appendix B: Proof of \cref{thm:low-rank-EIG}}
	\label{sec:AppendixB}

\section*{Appendix C: Low-rank approximation}
	\label{sec:AppendixC}
\end{comment}

\bibliographystyle{siamplain}
\bibliography{references}
\end{document}